\documentclass[11pt]{amsart}
\usepackage{geometry}
\usepackage{mathtools}
\usepackage{amssymb,amsthm,amsmath}
\usepackage[numbers,sort&compress]{natbib}
\usepackage{color}
\usepackage{graphicx}
\usepackage{tikz}
\usetikzlibrary{arrows}
\usepackage[colorlinks,linkcolor=blue,anchorcolor=blue,citecolor=blue]{hyperref}
\usepackage{comment}

\usepackage{enumerate}
\usepackage{enumitem}
\date{}

\title[]{Quantitative reducibility of Gevrey quasi-periodic cocycles and its applications}

\author{Xianzhe Li}
\address{Chern Institute of Mathematics and LPMC, Nankai University, Tianjin 300071, China} 
\email{xianzheli@mail.nankai.edu.cn}

\newtheorem{theorem}{Theorem}[section]
\newtheorem{proposition}{Proposition}[section]
\newtheorem{lemma}{Lemma}[section]
\theoremstyle{Definition}
\newtheorem{remark}{Remark}[section]
\newtheorem{Definition}{Definition}
\newtheorem{claim}{Claim}

\newcommand{\dif}{\mathrm{d}}   

\numberwithin{equation}{section}

\newcommand{\N}{{\mathbb N}}
\newcommand{\Q}{{\mathbb Q}}
\newcommand{\R}{{\mathbb R}}
\newcommand{\C}{\mathbb{C}}
\newcommand{\T}{{\mathbb T}}

\newcommand{\Z}{{\mathbb Z}}
\makeatletter 

\newcommand{\Rmnum}[1]{\expandafter\@slowromancap\romannumeral #1@}
\makeatother

\newcommand{\diam}{\mathrm{diam}}

\begin{document}

\begin{abstract}
    We establish a quantitative version of strong almost reducibility result for $\mathrm{sl}(2,\R)$ quasi-periodic cocycle close to a constant in Gevrey class. We prove that, for the quasi-periodic Schr\"odinger operators with small Gevrey potentials, the length of spectral gaps decays sub-exponentially with respect to its labelling,  the long range duality operator has pure point spectrum with sub-exponentially decaying eigenfunctions for almost all phases and the spectrum is an interval for  discrete Schr\"odinger operator acting on $ \Z^d $ with  small separable  potentials.  All these results are based on a refined KAM scheme, and thus are perturbative.
\end{abstract}
\maketitle
\section{Introduction}
We consider the quasi-periodic Schr\"odinger operator on $ \ell^2(\Z^b) $:
\begin{equation}\label{1.1}
    H_{v,\alpha,\theta}=\Delta+v(\theta+n\alpha)\delta_{nn'}, 
\end{equation} 
where $ \Delta $ is the usual Laplacian on $ \Z^b $ lattice, $v\in C^r(\mathbb{T}^d,\R) $ ($ r=0,1,\cdots,\infty,\omega $) is the \textit{potential}, and $ \alpha\in \R^d $, $ \theta\in\R^d $ are called
the \textit{frequency}, the \textit{phase} respectively. We will always assume that $ (1,\alpha) $ is rationally independent, thus the spectrum of $ H_{v,\alpha,\theta} $, denoted by $ \Sigma_{v,\alpha} $, is a compact set of $ \R $ with no isolated points, and independent on $ \theta $. Operator (\ref{1.1}) has been studied extensively for its rich background in quantum physics and quasi-crystal, one may consult \cite{you2018quantitative,ge2020arithmetic} and the references therein for partial advances. In this paper, we will focus on the properties of its spectrum. 

\subsection{Quantitative almost reducibility} The reducibility theory (consult section \ref{sec2.2} for related concepts) has attracted  much attention due to its importance in connecting dynamical systems and the spectral theory. 
The pioneering work was due to Dinaburg-Sinai \cite{dinaburg1975one}. They first used classical KAM theory to obtain positive measure energies has reducibility for one-dimensional continuous quasi-periodic Schr\"odinger equation, provided small analytic potential.  Eliasson \cite{eliasson} made a  breakthrough by proving the reducibility for almost all energies and the almost reducibility for all energies provided the frequency is Diophantine and the potential is small and  analytic, using a crucial resonance-cancelation technique introduced by  Moser-P\"oschel \cite{moser-p}.
The non-perturbative reducibility, i.e., the smallness of the perturbations does not depend on the Diophantine constant, was established by Hou-You \cite{hou} (see also Avila-Fayad-Krikorian \cite{krikorian2011kam}).
 Recently, Leguil-You-Zhao-Zhou \cite{leguil2017asymptotics} obtained a \textit{strong} version of quantitative almost reducibility result, i.e., the strips of width not going to zero. Almost reducibility have many interesting spectral applications, curious readers are invited to consult \cite{you2018quantitative}.

In all the above almost reducibility results, the perturbations are all assumed to be analytic. In this paper, we extend these results to Gevrey class (consult section \ref{sec2.1}). Recall that
$\alpha\in\R^d$ is called \textit{Diophantine} if there exist $ \gamma^{-1}>1$, $\tau>d$, such that $\alpha\in\mathrm{DC}_d(\gamma,\tau)$, 
 where
\[
    \mathrm{DC}_d(\gamma,\tau):=\left\{x\in\R^d:\inf_{j\in\Z}|\langle n,x\rangle- j|\geq\frac{\gamma}{|n|^\tau}, 0\neq n\in\mathbb{Z}^d\right\}.
\]
Let $ \mathrm{DC}_d:=\bigcup_{\gamma^{-1}>1,\tau>d}\mathrm{DC}_d(\gamma,\tau) $, then $ \mathrm{DC}_d $ has full Lebesgue measure. Then we can state our main result:
\begin{theorem}\label{thm1.1}
	Given $ r_0>0 $. Assume $ \alpha \in \mathrm{DC}_d(\gamma,\tau)$, $ A_0\in\mathrm{SL}(2,\R) $,  $f_0\in  G^\nu_{r_0}(\mathbb{T}^d,\mathrm{sl(2,\R)})$, with $ 0<\nu<1 $. Then for any $ r\in(0,r_0) $, there exists $ \epsilon_* $ depending on $ \gamma,\tau,d,\nu,r,r_0,\lVert A_0\rVert $, such that if $ |f_0|_{r_0}<\epsilon_* $, then the system $ (\alpha,A_0e^{f_0}) $ is strong almost reducible with the ``width'' $r$.
\end{theorem}
More detailed quantitative estimates will be given in section 4. As a comparison, we mention that though the strong almost reducibility for Gevrey class have been proved by Chavaudret \cite{chavaudret2013strong} in continuous analogy. We remark that our result provides more quantitative estimates, which plays a key  role in the following applications to quasi-periodic Schr\"odinger operators with Gevrey potentials.

\subsection{Estimates on spectral gaps}Recall that the bounded connected component of $ \R\backslash \Sigma_{v,\alpha} $ is called the \textit{spectral gap}, from the mathematical perspective, gap estimates is a core problem in the study of the spectral theory of quasi-periodic Schr\"odinger operators, since it can reflect the topological structure of the spectrum. 

For $ b=1 $, the famous Gap-Labelling theorem \cite{gaplabel} states that for any spectral gap $G$, there exists a unique $k\in \Z^d$ such that $2\rho \equiv\langle k, \alpha\rangle \mod \Z$, where $ \rho $ is its fibered rotation number (consult section \ref{sec2.2}). That is, all the spectral gaps can be labelled by integer vectors: we
denote by $G_k(v ) = (E^{-}_k , E^{+}_k )$ the gap with label $k\neq 0$.
When $E^{-}_k=E^+_k$, we
say the gap is collapsed.
We also set $E_0^{-}:= \inf \Sigma_{v,\alpha},E_0^+ := \sup \Sigma_{v,\alpha},$ and 
$G_0(v ) := (-\infty, E_0^{-}) \cup(E_0^+,\infty)$.	
The consideration of the estimates of gaps, may go back to a question that whether the spectrum is a Cantor set for the almost Mathieu operator (AMO), for which $ v(\cdot)=2\lambda\cos2\pi(\cdot) $, $ \lambda\neq 1 $ and $ \alpha\in\R\backslash\Q $. This is so-called \textit{Ten Martini Problem} stated by Barry Simon after an offer of Mark Kac in 1981 \cite{simon2000schrodinger}. After works of many mathematicians, this problem was finally solved by Avila-Jitomirskaya \cite{avila2009ten}. The \textit{Dry Ten Martini Problem} further conjectures that all gaps $ G_k $ labeled by the Gap-Labelling theorem are open, for AMO case. This problem was partially resolved in \cite{puig,avila2009ten,liu2015spectral,avila2016dry}. Moreover, Leguil-You-Zhao-Zhou \cite{leguil2017asymptotics} obtained the exponential asymptotic behavior on the gaps of noncritical AMO recently, which gives a quantitative estimate on their size. 

 It was conjectured that 
 one-dimensional quasi-periodic Schr\"odinger operator generically has Cantor spectrum \cite{simon1982almost}. For partial advances we refer to \cite{wang2019genericity}.
Then people are also curious about how large are the gaps. There are many works along this line:  For  analytic small potentials, Moser-P\"oschel \cite{moser-p} proved that  $ G_k(v) $ is exponentially small with respect to $ |k| $ provided that $ |k| $ large enough and $ \langle k,\alpha\rangle $ not too close to other $ \langle m,\alpha\rangle $; Amor \cite{amor} proved the all gaps $ G_k(v) $ are  sub-exponentially small with respect to $ |k| $  for small analytic potentials and Diophantine $\alpha$; Damanik-Goldstein \cite{damanik2014inverse} proved the  decay rate  $ |G_k(v)|\leq 2\lVert v\rVert_{r_0}e^{-\pi r_0|k|} $ for sufficiently small $ v\in C^\omega_{r_0}(\T^d,\R) $. Later on, Leguil-You-Zhao-Zhou \cite{leguil2017asymptotics} improved the decay rate as $ |G_k(v)|\leq \lVert v\rVert_{r_0}^{\frac{2}{3}}e^{-2\pi r|k|} $,  where  $ r\in(0,r_0) $ can be arbitrary close to $ r_0 $. For weakly coupled quasi-periodic Schr\"odinger operators with partial Liouville frequencies, see also the result by Liu-Shi \cite{liu2019upper}, the size of the spectral gaps decays exponentially. 

If the regularity of $ v $ is weaker than analytic, one can not expect an exponential decay rate \cite{damanik2014inverse}. 
In finite differentiable regime, Cai-Wang \cite{cai2021polynomial} showed that the gap lengths decay  polynomially. In this paper,  we improve this result to sub-exponential decay for small Gevrey potentials.
\begin{theorem}\label{thm1.2}
    Given $ r_0>0 $. Assume that $\alpha\in\mathrm{DC}_d(\gamma,\tau)$, $v\in  G^{\nu}_{r_0}(\mathbb{T}^d,\R)$ with $0<\nu<1$.
    For any $r\in(0,r_0)$, there exists $\epsilon_*=\epsilon_*(\gamma,\tau,d,\nu,r_0,r)>0$ such that if $|v|_{r_0}=\epsilon_0<\epsilon_*$, then
    for any $k\in\mathbb{Z}^d\backslash\{0\}$, the $k^{\mathrm{th}}$ gap of $H_{v,\alpha,\theta}$, denoted by $G_k(v)$, has the following estimate\[|G_k(v)|\leq\epsilon_0^{\frac{1}{2}}e^{- r|2\pi k|^{\nu}}.\] 
\end{theorem}
One should not expect all gaps are open for general quasi-periodic Schr\"odinger operators \cite{goldstein2019spectrum}, thus different from the AMO case, in general it is impossible to give any non-trivial lower bound of gap length.
 
\subsection{Long range operator}
The  Aubry duality of one dimensional  Schr\"odinger operators $ H_{\lambda^{-1} v,\alpha,\theta} $ are long-range operators on $ \Z^d $
\begin{equation} \label{longrange}
    (L_{\lambda^{-1} v,\alpha,\phi}u)_n=\sum_{k\in\Z^d}\hat{v}(n-k)u_k+2\lambda\cos 2\pi(\phi+\langle n,\alpha\rangle)u_n, \phi\in\T,
\end{equation} 
then we have 
\begin{theorem}\label{thmlong}
    If $ \alpha\in \mathrm{DC}_d$, $ v\in G_{r_0}^\nu(\T^d,\R) $, then for any $ r\in(0,r_0) $, there exists $ \lambda_0(\alpha,d,\nu,r_0,r) $, such that if $ \lambda>\lambda_0 $, $ L_{\lambda^{-1} v,\alpha,\phi} $ has pure point spectrum with sub-exponentially decaying eigenfunctions for a.e. $ \phi\in \T $.
\end{theorem}
If $ d=1 $, Bourgain-Jitomirskaya \cite{bourgain2002absolutely} showed that for any fixed Diophantine $ \alpha $, $ L_{\lambda^{-1} v,\alpha,\theta} $ has Anderson localization for sufficiently large $ \lambda $ and a.e. $ \theta $. This result is non-perturbative in the sense that the largeness of $ \lambda $ is independent of $ \alpha $. Recently, Ge-You-Zhou \cite{ge2019exponential} showed the same holds for $ d\geq2 $, but is a perturbative result, in the sense that the largeness of $ \lambda $ depends on $ \alpha $. Theorem \ref{thmlong} extends their result to Gevrey class. As a comparison, we mention that Shi \cite{shi2019spectral} proved that for multi-frequency quasi-periodic operator with Gevrey type perturbations, that is, 
$$
(L_{\lambda, v,f,\alpha,\phi}u)_n=\sum_{k\in\Z^d}\hat{v}(n-k)u_k+\lambda f(\phi+\langle n,\alpha\rangle)u_n, \phi\in\T^d,
$$ 
pure point spectrum with sub-exponentially decaying eigenfunctions holds for fixed $ \phi\in\T^d $, sufficiently large $ \lambda $ and a positive measure $ \alpha $-set, by using the techniques of LDT and Green's function estimates. While Theorem \ref{thmlong} is for fixed $ \alpha $, if $ \lambda $ sufficiently large (depends on $ \alpha $), we have pure point spectrum with sub-exponentially decaying eigenfunctions holds for a.e $ \phi $, we emphasize that the decay rate can be arbitrary close to $ r_0 $.
\subsection{Interval spectrum}In the study of spectral theory, an interesting question to ask is how complex the spectrum is.  If we consider the topological structure of a compact set, there are two extremes, one is the interval, one is the Cantor set. Interval is simple, intuitive, while the Cantor set is complex, counterintuitive. For $ b=1 $, it is known that if the frequency $ \alpha\in\Q^d $, the spectrum of $ H_{v,\alpha,\theta} $ is a disjoint union of at most finite compact intervals. Furthermore, Bethe–Sommerfeld Conjecture states that for $ b\geq 2 $, the spectrum of the $ b $-dimensional Schr\"odinger operator with periodic potential contains only finitely many gaps, partial advances see \cite{parnovski2008bethe,han2018discrete}. But for quasi-periodic case, the spectrum contains only finitely many gaps is a novel phenomenon. For $ b=1 $, Goldstein-Schlag-Voda \cite{goldstein2019spectrum} showed that for multi-frequency case, the spectrum can be a single interval. While for $ b\geq 2 $, things become easier. It is known that the spectrum of separable $ b $-dimensional Schr\"odinger operators are sums of the spectrum of single 1-dimensional operators (consult section \ref{separable}). Thus if we assume the potential $ v $ of $ H_{v,\alpha,\theta} $ is separable potentials,  intuitively, even though the spectrum of 1-dimensional operators is a Cantor set, if those gaps in every Cantor set have good control, then their sums can still produce an interval, thus it is natural to expect a condition such that the interval spectrum occurs. 
The famous Newhouse Gap Lemma \cite{newhouse1979abundance} provides a tool to link the gap information for a single Cantor set and their sums, thus our gap estimates can be applied to show the interval spectrum. This approach has been studied by Takase \cite{takase2021spectra}, who showed that if any 1-dimensional Schr\"odinger operator with analytic potential frequency-dependent close to constants, then their sums is an interval. We extend this work to infinite differentiable regime, that is, the following result:
\begin{theorem}\label{intervalspec}
	For separable $ b $-dimensional discrete Schro\"odinger operators $\hat{H}$ generated by one dimensional quasi-periodic Gevrey-class Schr\"odinger operators
	\begin{equation}
		\hat{H}=\Delta+\sum_{j=1}^b v_j(n_j\alpha_j+\theta_j)\delta_{nn'},
	\end{equation}
    where $ v_j\in G_{r_0}^{\nu}(\T^d,\R) \text{ with } 0<\nu<1$, for $ j=1,\cdots,b $. Assume $ \alpha_j \in \mathrm{DC}_{d}(\gamma,\tau) $, then there exists $\varepsilon=\varepsilon(\gamma,\tau,d,\nu,r_0)>0$, such that if 
    $$ 
    0<|v_j|<\varepsilon \text{ \ \ for } j=1,\cdots,b, 
    $$
    the spectrum of $ \hat{H} $ is an interval.
\end{theorem}
\begin{remark}
    Theorem \ref{intervalspec} is necessarily perturbative due to \cite{bourgain2002spectrum}.
\end{remark}
\subsection{Structure of the paper}
Some preliminaries are introduced in section 2. Quantitative almost reducibility for Gevrey class is established in section 3, section 4. In section 5, section 6, section 7, we will apply quantitative almost reducibility to finish the proof of Theorems \ref{thmlong}, \ref{thm1.2} and \ref{intervalspec}, respectively. Some standard proofs are included in the Appendix for completeness.

\section{Preliminaries}

For an $ n\times n $ matrix $A$, we will use the canonical norm $\|A\|:=n \max\limits_{1\leq i,j\leq n}|a_{ij}|$. 
\subsection{Gevrey functions}\label{sec2.1}
Recall that a function 
$ f\in C^{\infty}(\T^d,\R) $ 
is called $ \nu $-Gevrey regular, if we write it as $ f(\theta)=\sum_{k\in\Z^d}\widehat{f}(k)e^{\mathrm{i}2\pi \langle k,\theta\rangle} $, there exists $ \nu\in(0,1] $, such that $ |\widehat{f}(k)|\lesssim e^{-\rho |k|^{\nu}} $, $ \forall k\in\Z^d $, where $ \rho>0 $, $ |k|=\sum_{i=1}^d |k_i| $. In particular, if $ \nu=1 $, $ f $ is analytic. In this paper, we will set
\[ |f|_r=|f|_{\nu,r}:=\sum_{k\in\Z^d}|\widehat{f}(k)|e^{r|2\pi k|^{\nu}}<\infty, \ 0<\nu<1, \]
where $ r>0 $ is called the ``width'', and denote by $ G^\nu_r(\T^d,*) $ the set of these $ * $-valued functions ($ * $ will usually denote $ \R,\mathrm{SL}(2,\R)$).   
\begin{remark}
    If $ f\in G_{r}^\nu(2\T^d,\R) $, $ |f|_{r}:=\sum_{k\in\Z^d}|\widehat{f}(k)|e^{r|\pi k|^\nu}. $
\end{remark}
\subsection{Quasi-periodic cocycle and reducibility}\label{sec2.2}
In this subsection, we recall some basic concepts in quasi-periodic dynamics.

For given $ A\in C^0(\T^d,\mathrm{SL}(2,\C)) $ and $ \alpha\in\R^d $ with $ (1,\alpha) $ rationally independent, the quasi-periodic \textit{cocycle} is defined as
$$
    (\alpha,A):\left\{ \begin{aligned}
        \T^d\times\C^2& & &\to & &\T^d\times \C^2\\
        (\theta,v)& & &\mapsto & &(\theta+\alpha,A(\theta)v)
    \end{aligned}\right. .
$$ 
Denote $ (\alpha,A)^n=(n\alpha,\mathcal{A}_n) $, where
$$
    \mathcal{A}_n(\theta):=\left\{\begin{aligned}
        &A(\theta+(n-1)\alpha)\cdots A(\theta+\alpha)A(\theta), & &n\geq 0\\
        &A^{-1}(\theta+n\alpha)A^{-1}(\theta+(n+1)\alpha)\cdots A^{-1}(\theta-\alpha), & &n<0
    \end{aligned}\right. .
$$ 
Then the \textit{Lyapunov exponent} is well-defined by $$ L(\alpha,A):=\lim_{n\to\infty}\frac{1}{n}\int_{\T^d}\ln\lVert\mathcal{A}_n(\theta)\rVert\dif \theta .$$

The cocycle $ (\alpha,A) $ is \textit{uniformly hyperbolic} if for any $ \theta\in\T^d $, there exists a continuous splitting $ \C^2=E^s(\theta)\oplus E^u(\theta) $ such that for every $ n\geq 0 $,
$$
    \begin{aligned}
        \lvert \mathcal{A}_n(\theta)v\rvert&\leq Ce^{-cn}\lvert n\rvert, & &v\in E^s(\theta),\\
        \lvert \mathcal{A}_n(\theta)^{-1}v\rvert&\leq Ce^{-cn}\lvert n\rvert, & &v\in E^u(\theta+n\alpha),
    \end{aligned}
$$ 
for some constants $ C,c>0 $, i.e., the splitting is invariant by the dynamics. 
If we consider the eigenvalue equation $H_{v,\alpha, \theta}u=Eu$, then we can induce a \textit{Schr\"odinger cocycle}, denoted by $(\alpha,S_E^v)$:
\begin{equation}
    \begin{pmatrix}
        u_{n+1}\\u_n
    \end{pmatrix}=
    \begin{pmatrix}
        E-v(\theta+n\alpha)&-1\\
        1&0
    \end{pmatrix}
    \begin{pmatrix}
        u_n\\u_{n-1}
    \end{pmatrix}.
\end{equation}
It is well-known that $ E\notin \Sigma_{v,\alpha} $ if and only if $ (\alpha, S_E^v) $ is uniformly hyperbolic.

Assume $ A\in C^0(\T^d,\mathrm{SL}(2,\R)) $ is homotopic to the identity, then there exist $ \psi:\T^d\times\T\to\R $ and $ u:\T^d\times\T\to\R_+ $ such that 
$$
    A(x)\cdot\begin{pmatrix}
        \cos 2\pi y\\\sin 2\pi y
    \end{pmatrix}=u(x,y)\begin{pmatrix}
        \cos 2\pi (y+\psi(x,y))\\\sin 2\pi (y+\psi(x,y))
    \end{pmatrix}.
$$ 
The function $ \psi $ is called a \textit{lift} of $ A $. Let $ \mu $ be any probability measure on $ \T^d\times\R $ which is invariant by the continuous map $ T:(x,y)\mapsto(x+\alpha,y+\psi(x,y)) $, projecting over Lebesgue measure on the first coordinate. 
Then the number
$$
    \rho(\alpha,A)=\int\psi\ \dif\mu \mod \Z
$$ 
does not depend on the choices of $ \psi $ and $ \mu $ and is called the \textit{fibered rotation number} of $ (\alpha,A) $, see \cite{gaplabel}. Given $ \phi\in\T^d $, let $ R_{\phi}:=\begin{pmatrix}
    \cos 2\pi \phi&\sin 2\pi \phi\\
    \cos 2\pi \phi&\cos 2\pi \phi
\end{pmatrix} $. It is immediate from the definition that
\begin{equation}\label{lem3.2} 
    |\rho(\alpha,A)-\phi|<\lVert A-R_\phi\rVert_{0},
\end{equation} 
where $ \lVert\cdot\rVert_0 $ denotes the $ C^0 $ norm.
Any $ B:2\T^d\to\mathrm{SL}(2,\R) $ is homotopic to $ \theta\in\T^d\to R_{\frac{\langle n,\theta\rangle}{2}} $ for some $ n\in\Z^d $ called the \textit{degree} of $ B $, denoted by $ \deg B=n $. 
Given two real cocycles $ (\alpha,A) $ and $ (\alpha,A') $, we say they are \textit{real conjugated} to each other if there exists $ B\in C^0(2\T^d, \mathrm{SL}(2,\R)) $ such that
$$
    B(\theta+\alpha)^{-1}A(\theta)B(\theta)=A'(\theta).
$$ 
The fibered rotation number has the following property:
\begin{theorem}[\cite{krikorian2004}]
	Assume $ (\alpha,A) $ is conjugated to $ (\alpha, A') $ by $ B\in C^0 (2\mathbb{T}^d,\mathrm{SL}(2, \R)) $ with $ \deg B=n $, then we have
    \begin{equation}\label{degree}
	\rho(\alpha, A)=\rho(\alpha, A')+\frac{\langle n,\alpha\rangle}{2}.
	\end{equation}
\end{theorem}

The cocycle $ (\alpha,A) $ is said to be ($ \nu $-Gevrey) \textit{reducible} if it can be conjugated to a constant cocycle, that is, there exist $ B\in G^\nu(2\mathbb{T}^d,\mathrm{SL}(2,\R)) $ and $ A_0\in\mathrm{SL}(2,\R) $ such that
\[
    B(\theta+\alpha)^{-1}A(\theta)B(\theta)=A_0.
\]
It is said to be \textit{almost reducible}, if the closure of its conjugates contains a constant cocycle.

\subsection{Separable potential}\label{separable}
Let $ v:\Z^d\to\R $ be a potential. We call $ v $ is \textit{separable} if $ v $ is given by
$$
    v(n)=v_1(n_1)+\cdots+v_d(n_d),
$$ 
where $ n=(n_1,\cdots,n_d)\in\Z^d $.
Consequently, we have
\begin{theorem}[\cite{damanik2011spectral}]\label{separablespec}
    For d-dimensional discrete Schro\"odinger operators on $ \ell^2(\Z^d) $
    \begin{equation*}
		H=\Delta+\sum_{j=1}^d v_j(n_j)\delta_{nn'},
	\end{equation*}
    denote its spectrum by $ \hat{\Sigma} $. Consider the associated 1-dimensional Schr\"odinger operators on $ \ell^2(\Z) $
    $$
        (H_j u)_n=\Delta+v_j(n)\delta_{nn'},
    $$ 
    and denote its spectrum by $ \Sigma_j $. Then we have
    $$
        \hat{\Sigma}=\Sigma_1+\cdots+\Sigma_d.
    $$ 
\end{theorem}
\subsection{Thickness and the Newhouse Gap Lemma}
Let $ K $ be a nonempty compact subset of $ \R $, we define its \textit{thickness}. We consider the \textit{gaps} of $ K $: a gap of $ K $ is a connected component of $ \R\backslash K $; a bounded gap is a bounded connected component of $ \R\backslash K $. If $ I $ is an interval, let $ \ell(I) $ be its \textit{length}. Let $ U $ be any bounded gap and $ u $ be a boundary point of $ U $, thus $ u\in K $. Let $ C $ be the \textit{bridge} of $ K $ at $ u $, i.e., the maximal interval in $ \R $ such that\begin{enumerate}
    \item $ u $ is a boundary point of $ C $,
    \item if $ C $ contains points of a gap $ V $, then $ \ell(V)<\ell(U) $.
\end{enumerate}

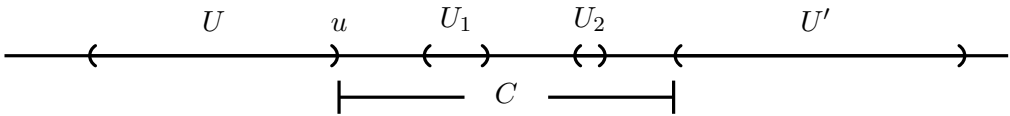
\begin{figure}[htbp]
    \centering
    \begin{tikzpicture}[scale=11]
    \draw[-, very thick] (-0.1,0) -- (1.1,0);
    \draw[(-), very thick] (0,0) -- (0.3,0);
    \draw[(-), very thick] (0.4,0) -- (0.48,0);
    \draw[(-), very thick] (0.58,0) -- (0.62,0);
    \draw[(-), very thick] (0.7,0) -- (1.05,0);
    \draw[-, very thick] (0.3,-0.05) -- (0.45,-0.05);
    \draw[very thick] (0.3,-0.03)--(0.3,-0.07);
    \draw[ -, very thick] (0.55,-0.05) -- (0.7,-0.05);
    \draw[very thick] (0.7,-0.03)--(0.7,-0.07);
    \draw (0.3,0.5pt) node[above] {$u$};
    \draw (0.15,0.5pt) node[above] {$U$};
    \draw (0.44,0.5pt) node[above] {$U_1$};
    \draw (0.6,0.5pt) node[above] {$U_2$};
    \draw (0.87,0.5pt) node[above] {$U'$};
    \draw (0.5,0.5pt) node[below=12pt] {$ C $};
    \end{tikzpicture}  
    \caption{$ \ell(U')>\ell(U) $ and $ C $ is the bridge of $ K $ at $ u $}
\end{figure}

Then the thickness of $ K $ at $ u $ is defined as $ \tau(K,u)=\ell(C)/\ell(U) $ and the thickness of $ K $, denoted by $ \tau(K) $, is the infimum over these $ \tau(K,u) $ for all boundary points $ u $ of bounded gaps. 

Note that $ \tau(K)=0 $, if $ K $ has an isolated point, and $ \tau(K)=+\infty $ when $ K $ is an interval.
To get interval spectrum, we require the following Gap Lemma in \cite{astels2000cantor}, which is the generalization of the classical Gap Lemma established by Newhouse \cite{newhouse1979abundance}.
\begin{theorem}[\cite{astels2000cantor}]\label{astels}
    Let $ K_1,\cdots,K_d$  $ (d\ge 2) $ be nonempty compact subsets of $ \R $. For each nonempty compact subset $ K $ of $ \R $, define 
    $$ 
    \Gamma(K):=\sup\{\ell(U):\text{ U is a bounded gap of }  K \}. 
    $$ 
    Assume
    \begin{equation} \label{GammaK}
        \left\lbrace\begin{aligned}
            &\Gamma(K_j)\leq \diam(K_i),\ \ \forall 1\leq j<i\leq d,\\
            &\Gamma(K_i)\leq \diam(K_1)+\cdots+\diam(K_{i-1}),\ \ \forall 2\le i\le d.
        \end{aligned}\right.
    \end{equation} 
    Then if $ \sum\limits_{i=1}^d\frac{\tau(K_i)}{\tau(K_i)+1}\geq 1 $, we have $ \tau(K_1+\cdots+K_d)=+\infty $ and
        $
            K_1+\cdots+K_d
        $ is an interval.

\end{theorem}
\begin{remark}
    In \cite{astels2000cantor,newhouse1979abundance}, the author used another definition of thickness, the equivalence of these two definitions can be found in \cite{palis1995hyperbolicity}, p. 61.   
\end{remark}
\section{The (non-)resonant space}
In Hou-You \cite{hou}, a crucial observation is that one can eliminate the non-resonant terms. In this section, we will prove a similar lemma to eliminate the non-resonant terms which is suitable for our Gevrey case.

For any $N>0$,  we define the truncation operator as\[(\mathcal{T}_N f)(\theta):=\sum_{k\in\mathbb{Z}^d,|k|\leq N} \widehat{f}(n)e^{\mathrm{i}2\pi\langle n,\alpha\rangle},\]
and for any set $K\subset{\Z^d}$, we denote \[(\mathcal{S}_K f)(\theta):=\sum_{k\in K}\widehat{f}(k)e^{\mathrm{i}2\pi\langle k,\alpha\rangle}.\]
 Given $\eta>0$, $\alpha\in\R^d$, and $ A \in \mathrm{SL}(2,\R) $. 
  Recall that there is an isomorphism from $ \mathrm{sl}(2,\R) $ to $ \mathrm{su}(1,1)$, which is given by $B\mapsto MBM^{-1}$ with \[M=\frac{1}{1+ \mathrm{i}}\begin{pmatrix}
    	1&- \mathrm{i}\\1& \mathrm{i}
    	\end{pmatrix}.\]
    	More precisely, a direct computation leads to
        \begin{equation}\label{isom}
    	M\begin{pmatrix}
    		x&y+z\\y-z&-x
    		\end{pmatrix}M^{-1}=\begin{pmatrix}
    		\mathrm{i}z&x-\mathrm{i}y\\x+\mathrm{i}y&-\mathrm{i}z
    		\end{pmatrix}.
    	\end{equation}
 Assume that the two eigenvalues of $ A $ are $ e^{\pm\mathrm{i}2\pi\xi} $, let 
 $$ 
 \Phi:=\{n\in\Z^d\backslash\{0\}:|e^{\mathrm{i}2\pi(\langle n,\alpha \rangle\pm2\xi)}-1|\geq\eta, |e^{\mathrm{i}2\pi\langle n,\alpha\rangle}-1|\geq \eta\} 
 $$ and $\Phi^c:\Z^d-\Phi$. Define the following subspace, 
 \begin{align*}
    \widetilde{G}_{r}^{\nu\,(nre)}(\eta)&=\left\lbrace \left.\begin{pmatrix}
 \mathrm{i}\mathcal{S}_{\Phi}f&\mathcal{S}_{\Phi} g\\
 \mathcal{S}_{\Phi} \bar{g}&-\mathrm{i}\mathcal{S}_{\Phi}f
 \end{pmatrix}
 \right\vert f\in G^\nu_r(\mathbb{T}^d,\R), g\in G^\nu_r(\mathbb{T}^d,\mathbb{C})\right\rbrace,\\
 \widetilde{G}_{r}^{\nu\,(re)}(\eta)&=\left\lbrace \left.\begin{pmatrix}
 \mathrm{i}\mathcal{S}_{\Phi^c}f&\mathcal{S}_{\Phi^c} g\\
 \mathcal{S}_{\Phi^c} \bar{g}&-\mathrm{i}\mathcal{S}_{\Phi^c}f
 \end{pmatrix}
 \right\vert f\in G^\nu_r(\mathbb{T}^d,\R), g\in G^\nu_r(\mathbb{T}^d,\mathbb{C})\right\rbrace,
\end{align*} 
and also define $G_{r}^{\nu\,(nre)}(\eta)=M^{-1}\widetilde{G}_{r}^{\nu\,(nre)}(\eta)M$, $G_{r}^{\nu\,(re)}(\eta)=M^{-1}\widetilde{G}_{r}^{\nu\,(re)}(\eta)M$, we call $ G_{r}^{\nu\,(nre)}(\eta) $ the non-resonant space and $ G_{r}^{\nu\,(re)}(\eta) $ the resonant space, all these subspaces above are Banach space. It follows that 
\begin{align*}
    G^\nu_r(\mathbb{T}^d,\mathrm{su(1,1)})&=\widetilde{G}_{r}^{\nu\,(re)}(\eta)\oplus \widetilde{G}_{r}^{\nu\,(nre)}(\eta),\\
    G^\nu_r(\mathbb{T}^d,\mathrm{sl(2,\R)})&=G_{r}^{\nu\,(re)}(\eta)\oplus G_{r}^{\nu\,(nre)}(\eta).
\end{align*}
Under this decomposition, we have the following lemma:
\begin{lemma}\label{lem3.1}
	Assume that $ A\in \mathrm{SL(2,\R)} $, there exist numerical constants $ c_1,D_1 $ such that if 
    $$ \epsilon\leq c_1 \lVert A\rVert^{-D_1}\ \text{and}\ \ \eta\geq \|A\|^2\epsilon^{\frac{1}{9}},$$
     then for any $ f\in  G_r^\nu(\mathbb{T}^d,\mathrm{sl(2,\R)}) $ with $ |f|_{\Lambda,r}\leq\epsilon $, there exist $ Y\in G_{r}^{\nu\,(nre)}(\eta) $ and $ f^{(re)}\in G_{r}^{\nu\,(re)}(\eta) $ such that 
     \[ 
         e^{-Y(\theta+\alpha)}(Ae^{f(\theta)})e^{Y(\theta)}=Ae^{f^{(re)}(\theta)},
    \]
    with $ |Y|_{r}\leq \epsilon^{\frac{1}{2}} $, and $ |f^{(re)}|_{r}\leq 2\epsilon $.
\end{lemma}
\begin{remark}
The analytic version of this result has appeared in \cite{hou,Cai2017SharpHC}. The proof of this lemma is based on a quantitative Implicit Function Theorem, 
and will be given in the Appendix \ref{appendixA} for completeness.
\end{remark}

\section{Quantitative almost reducibility}
\subsection{The inductive step}
Consider the following quasi-periodic $\mathrm{SL}(2,\R)$ cocycle \[\begin{pmatrix}
u_{n+1}\\u_n
\end{pmatrix}=(Ae^{f(\theta+n\alpha)})\begin{pmatrix}
u_n\\u_{n-1}
\end{pmatrix},\] 
with $\alpha\in\mathrm{DC}_d(\gamma,\tau)$, $ A\in\mathrm{SL}(2,\R) $, $ f\in G^{\nu}_r(\mathbb{T}^d,\mathrm{sl(2,\R)}) $, for some $0<\nu<1$, denote it by $(\alpha,Ae^{f})$. Then we have
\begin{proposition}\label{prop}
	Let $\alpha\in\mathrm{DC}_d(\gamma,\tau)$, $\sigma=\frac{1}{20}$. Given $r>0$, 
    then for any $r_+\in(0,r)$, there exist $c=c(\gamma,\tau,d,\nu)>0 $ and a numerical constant $ D $ such that if 
    \begin{equation}\label{epsilon}
        |f|_r<\epsilon\leq c\lVert A \rVert^{-D}(r-r_+)^{D\nu\tau}, 
    \end{equation}
    then there exist $ B(\theta)\in G^{\nu}_{r_+}(2\mathbb{T}^d,\mathrm{SL(2,\R)}) $, $ A_+\in\mathrm{SL(2,\R)} $, $ f_+\in G^{\nu}_{r_+}(\mathbb{T}^d,\mathrm{sl(2,\R)}) $ such that
    $$
        B(\theta+\alpha)^{-1}Ae^{f(\theta)}B(\theta)=A_+e^{f_+(\theta)}.
    $$ 
    More precisely, assume the two eigenvalues of $ A $ are $ e^{\pm\mathrm{i}2\pi\xi} $, and let $ N:=\frac{1}{2\pi}\left( \frac{2|\ln\epsilon|}{r-r_+}\right)^{\frac{1}{\nu}}$, we have the following: 
    \begin{enumerate}[label=$\mathrm{\Roman*}$\ ,align=left,widest=ii,labelsep=0pt,leftmargin=0.5em]
		\item $\mathrm{(Non}$-$\mathrm{resonant\ case)}$ If for any $ n\in\Z^d $ with $0<|n|\leq N $,
        \begin{equation}\label{non}
		|2\xi-\langle n,\alpha\rangle|_{\T}\geq \epsilon^{\sigma},\ 
		\end{equation} 
        where $ |\cdot|_{\T}=\inf_{j\in\Z}|\cdot-j| $, then we have \[ \ \|A_+-A\|\leq 4\|A\|\epsilon,\ |f_+|_{r_+}\leq \epsilon^{2},\ |B-\mathrm{Id}|_{r_+}\leq 2\epsilon^{\frac{1}{2}}. \]
		\item $(\mathrm{Resonant\ case})$ If there exists $ n_*\in\Z^d $ with $ 0<|n_*|\leq N $ such that
        \begin{equation}\label{res}
		|2\xi-\langle n_*,\alpha\rangle|_{\T}<\epsilon^{\sigma},
		\end{equation} then we have $A_+=e^{a_+}$ with $\|a_+\|\leq 4\epsilon^{\sigma}$, and
        \begin{align*}
            &|f_+|_{r_+}\leq\epsilon^{\frac{3}{4}-\frac{2 r_+}{r-r_+}} e^{-(r-r_+) \epsilon^{-\frac{10\nu}{11}\frac{\sigma}{\tau}}}\ll\epsilon^{2}, \\&|B|_{r''}\leq 32 \gamma^{-\frac{1}{2}}\|A\|^{\frac{1}{2}} |n_*|^{\frac{\tau}{2}}e^{r''(\pi |n_*|)^{\nu}},\ \forall r''\in(0, r_+],
        \end{align*}
		and $\deg B=n_*$. Moreover, let $ a_+=M^{-1}\begin{pmatrix}
            \mathrm{i}t_+ & \mu_+\\
            \overline{\mu_+} & -\mathrm{i}t_+   
        \end{pmatrix}M $, we have
        \begin{equation} \label{3.4}
            |t_+|\leq \epsilon^{\sigma}, |\mu_+|\leq \epsilon^{\frac{3}{4}}e^{-r|2\pi n_*|^{\nu}}.
        \end{equation} 
	\end{enumerate}
\end{proposition}

    \begin{proof}
        For $\mathrm{I}$. Let $\eta=2\|A\|^2\epsilon^{\frac{1}{9}}$, then we have the corresponding decomposition, by (\ref{epsilon}), we can apply Lemma \ref{lem3.1}, then $ (\alpha,Ae^f) $ is conjugated to a cocycle $ (\alpha,Ae^g) $ with $ g\in G_{r}^{\nu\,(re)}(\eta) $. 

        The fact $ \alpha\in\mathrm{DC}_d(\gamma,\tau) $ and (\ref{epsilon}) can imply \begin{equation}\label{diophantine}
            |e^{\mathrm{i}2\pi\langle n,\alpha\rangle}-1|\geq\frac{\gamma}{N^\tau}\geq\epsilon^{\sigma}\geq\eta,\qquad \forall n\in\Z^d \text{ with } 0<|n|\leq N.	
            \end{equation} 
            Thus combining (\ref{non}) with (\ref{diophantine}), by the definition of $ G_{r}^{\nu\,(re)}(\eta) $, we have
            $$ \mathcal{T}_N g(\theta)-\widehat{g}(0)=0,$$  
            hence  
            \begin{align*}
                |g-\widehat{g}(0)|_{r_+}=\sum_{|n|>N}|\widehat{g}(n)|e^{  r_+ (2\pi |n|)^{\nu}}
                \leq |g|_r e^{-(r-r_+)(2\pi N)^{\nu}}\leq 2\epsilon^{3}.   
            \end{align*}
            Let $ B:=e^Y$, $A_+:=Ae^{\widehat{g}(0)} $, and $ e^{f_+}:=e^{-\widehat{g}(0)}e^{g} $. Clearly, we have the corresponding estimates.

        For $ \mathrm{II} $. Since $ \alpha\in\mathrm{DC}_d(\gamma,\tau) $, thus we only need to consider the case in which $ A $ is elliptic, i.e., $ \xi\in\R\backslash\{0\} $. In fact, if $\xi\in\mathrm{i}\R$, we must have (\ref{non}) holds. 
        \begin{lemma}
            For $ A $ is elliptic, there exists $P\in\mathrm{SL}(2,\R) $  such that 
            $$
            A=P^{-1}M^{-1}\begin{pmatrix}
                e^{\mathrm{i}2\pi\xi}&0\\0&e^{-\mathrm{i}2\pi\xi}
                \end{pmatrix}MP
            $$  with $\|P\|\leq 16\max\{\frac{1}{2},\sqrt{\frac{\|A\|}{2|\sin2\pi\xi|}}\}$.
        \end{lemma} 
        \begin{proof}
            By Schur's Lemma, there exists a unitary matrix $U$ such that 
            \[A=e^{a}=e^{U\begin{pmatrix}
                \mathrm{i}2\pi\xi&p\\0&-\mathrm{i}2\pi\xi
                \end{pmatrix}U^*}=U\begin{pmatrix}
                e^{\mathrm{i}2\pi\xi}&\frac{\sin2\pi\xi}{2\pi\xi}p\\0&e^{-\mathrm{i}2\pi\xi}
                \end{pmatrix}U^*.
            \]
            Notice that $ \lVert U\rVert\leq 2 $, we have $ \lVert a \rVert\leq 8\max\{|2\pi\xi|,|p|\} $, and $ |\frac{\sin 2\pi\xi}{2\pi\xi}||p|\leq 4\lVert A\rVert $, then by Lemma 8.1 in \cite{hou}, there exists $ P\in\mathrm{SL}(2,\R) $ with $ \lVert P\rVert\leq 2\sqrt{\frac{\lVert a\rVert}{|2\pi\xi|}} $ such that $ P a P^{-1}=\begin{pmatrix}
                0&2\pi\xi\\-2\pi\xi&0
            \end{pmatrix} $, then we have
            $$
                PAP^{-1}=M^{-1}\begin{pmatrix}
                    e^{\mathrm{i}2\pi\xi}&0\\0&e^{-\mathrm{i}2\pi\xi}
                \end{pmatrix}M
            $$ 
            with $ \lVert P\rVert\leq 16\max\{\frac{1}{2},\sqrt{\frac{\lVert A\rVert}{2|\sin2\pi\xi|}}\} $.                
        \end{proof}

        Let $ \tilde{A}:=M^{-1}\begin{pmatrix}
        e^{\mathrm{i}2\pi\xi}&0\\0&e^{-\mathrm{i}2\pi\xi}
        \end{pmatrix}M, \tilde{f}:=PfP^{-1} $, by (\ref{epsilon}) and (\ref{res}),
        \[ 
            |2\xi|_{\T}\geq |\langle  n_*,\alpha\rangle|_{\T}-\epsilon^\sigma\geq\frac{\gamma}{|n_*|^\tau}-\epsilon^\sigma\geq\frac{\gamma}{2|n_*|^\tau},
        \]  
        then by $ |\sin2\pi\xi|\geq \frac{2}{\pi}|2\xi|_{\T} $, we have
        \[
            |\tilde{f}|_r \leq \|P\|^2\epsilon:=\tilde{\epsilon}\leq 256\max\{\frac{1}{4},\gamma^{-1}\|A\||n_*|^{\tau}\}\epsilon\leq 256\gamma^{-1}\|A\||n_*|^{\tau}\epsilon\leq\epsilon^{\frac{8}{9}}.
        \]
        By the concrete algebraic structure of $\tilde{A}$, if we set 
        $$ 
        \begin{aligned}
            &\Phi_1:=\{n\in\Z^d:|e^{\mathrm{i}2\pi\langle n,\alpha\rangle}-1|\geq \eta\},\\
            &\Phi_2:=\{n\in\Z^d:|e^{\mathrm{i}2\pi(2\xi-\langle n,\alpha\rangle)}-1|\geq \eta\},
        \end{aligned}
        $$ 
        and redefine
     \begin{align*}
        \widetilde{G}_{r}^{\nu\,(nre)}(\eta)&=\left\lbrace \left.\begin{pmatrix}
     \mathrm{i}\mathcal{S}_{\Phi_1}f&\mathcal{S}_{\Phi_2} g\\
     \mathcal{S}_{-\Phi_2} \bar{g}&-\mathrm{i}\mathcal{S}_{\Phi_1}f
     \end{pmatrix}
     \right\vert f\in G^{\nu}_{r}(\mathbb{T}^d,\R), g\in G^{\nu}_{r}(\mathbb{T}^d,\mathbb{C})\right\rbrace,\\
     \widetilde{G}_{r}^{\nu\,(re)}(\eta)&=\left\lbrace \left.\begin{pmatrix}
     \mathrm{i}\mathcal{S}_{\Phi_1^c}f&\mathcal{S}_{\Phi_2^c} g\\
     \mathcal{S}_{-\Phi_2^c} \bar{g}&-\mathrm{i}\mathcal{S}_{\Phi_1^c}f
     \end{pmatrix}
     \right\vert f\in G^{\nu}_{r}(\mathbb{T}^d,\R), g\in G^{\nu}_{r}(\mathbb{T}^d,\mathbb{C})\right\rbrace,
    \end{align*} 
    $G_{r}^{\nu\,(nre)}(\eta)$, $G_{r}^{\nu\,(re)}(\eta)$ are redefined similar, we get the corresponding decomposition. Then we have the following which is similar with Lemma \ref{lem3.1}.
    \begin{lemma}\label{lem2.2}
        Assume $ A\in \mathrm{SO(2,\R)} $, there exist numerical constants $ c_2,D_2 $ such that if 
        $$ \epsilon\leq c_2\lVert A\rVert^{-D_2} \ \text{and}\ \ \eta\geq \|A\|^2\epsilon^{\frac{1}{3}},$$
        then for any $ f\in  G^{\nu}_{r}(\mathbb{T}^d,\mathrm{sl(2,\R)}) $ with $ |f|_{r}\leq\epsilon $, there exist $ Y\in G_{r}^{\nu\,(nre)}(\eta) $ and $ f^{(re)}\in G_{r}^{\nu\,(re)}(\eta) $ such that \[ e^{-Y(\theta+\alpha)}(Ae^{f(\theta)})e^{Y(\theta)}=Ae^{f^{(re)}(\theta)},\]with $ |Y|_{r}\leq \epsilon^{\frac{1}{2}} $, and $ |f^{(re)}|_{r}\leq 2\epsilon $.
    \end{lemma}
    
    \begin{proof}
        The only difference between Lemma \ref{lem3.1} and Lemma \ref{lem2.2} is that if $A\in SO(2,\R)$, for any $Y\in\mathrm{sl}(2,\R) $, $A^{-1}Y(\cdot+\alpha)A-Y(\cdot)$ still belongs to $\mathrm{sl}(2,\R)$. By the definition of redefined non-resonant space, one can easily check that 
        \[|A^{-1}Y'(\theta+\alpha)A-Y'(\theta)|_r\geq\eta|Y'|_r.\]The rest follows a same line as Lemma \ref{lem2.2}.
    \end{proof}
    
        Let $ \eta:=2\|\tilde{A}\|^2\tilde{\epsilon}^\frac{1}{3} $, then one can apply Lemma \ref{lem2.2}, thus $(\alpha,\tilde{A}e^{\tilde{f}})$ can be conjugated to $(\alpha,\tilde{A}e^g)$ with $g\in  G_{r}^{\nu\,(re)}(\eta)$ by a conjugacy $e^Y$,  and $|Y|_r\leq \tilde{\epsilon}^{\frac{1}{2}} $, $|g|_r\leq 2\tilde{\epsilon}.$
       By the definition of $G_{r}^{\nu\,(re)}(\eta)$, there exist $ g_{11}(\theta)\in G^\nu_{r}(\mathbb{T}^d,\R) $, $ g_{12}(\theta)\in G^\nu_{r}(\mathbb{T}^d,\mathbb{C}) $ such that 
       \begin{align}\label{gre}
       Mg(\theta)M^{-1} &=\begin{pmatrix}
     \mathrm{i}\mathcal{S}_{\Phi_1^c} g_{11}(\theta),\mathcal{S}_{\Phi_2^c} g_{12}(\theta)\\
     \mathcal{S}_{-\Phi_2^c} \overline{g_{12}}(\theta),-\mathrm{i}\mathcal{S}_{\Phi_1^c} g_{11}(\theta)
     \end{pmatrix}\notag\\
     &=\sum_{n\in \Phi_1^c}\begin{pmatrix}
            \mathrm{i} \widehat{g_{11}}(n)&0\\0&-\mathrm{i}\widehat{g_{11}}(n)
            \end{pmatrix}e^{\mathrm{i}2\pi\langle n,\theta\rangle}
            \\&+\sum_{n\in \Phi_2^c}\begin{pmatrix}
            0&\widehat{g_{12}}(n)e^{\mathrm{i}2\pi\langle n,\theta\rangle}\\ \overline{\widehat{g_{12}}(n)}e^{-\mathrm{i}2\pi\langle n,\theta\rangle}&0
            \end{pmatrix}.
       \end{align}

    By the Diophantine condition and (\ref{res}), we have the following observation:
    \begin{claim} \label{claim}
        \begin{equation*}\label{22}
        \Phi_1^c\cap\{n\in\Z^d:|n|\leq N\}=\{0\},\Phi_2^c\cap\{n\in\Z^d:|n|\leq N\}=\{n_*\}.
        \end{equation*}
    \end{claim}
    \begin{proof}
    For any $n\neq0$ and $n\in\Phi_1^c$, we have\[\frac{\gamma}{|n|^\tau}\leq|e^{\mathrm{i}2\pi\langle n,\alpha\rangle}-1|<\eta\leq\epsilon^\sigma, \]thus $|n|> \gamma^{\frac{1}{\tau}}\epsilon^{-\frac{\sigma}{\tau}}.$ For any $ n\in\Phi_2^c $, we have \[|2\xi-\langle n,\alpha\rangle|_{\T}<\frac{1}{4}\eta.\]	 		
    Assume there is an $ n_*'\neq n_* $ with $ |2\xi-\langle n_* ',\alpha\rangle|_{\T} <\frac{1}{4}\eta$, then we have 
    \[ 
        \frac{\gamma}{|n_*'-n_*|^\tau}\leq|e^{\mathrm{i}2\pi\langle n_*'-n_*,\alpha\rangle}-1|\leq 2|\langle n_*'-n_*, \alpha\rangle|_{\T}<2\epsilon^{\sigma}, 
    \]
    which implies $ |n_*'|>2^{-\frac{1}{\tau}}\gamma^{\frac{1}{\tau}}\epsilon^{-\frac{\sigma}{\tau}}-N $.
    Let $ N_1:=\gamma^{\frac{1}{\tau}}\epsilon^{-\frac{\sigma}{\tau}} $, $ N_2:=2^{-\frac{1}{\tau}}\gamma^{\frac{1}{\tau}}\epsilon^{-\frac{\sigma}{\tau}}-N $, then by (\ref{epsilon}), we have $\epsilon^{-\frac{10\sigma}{11\tau}}\leq 2^{-1-\frac{1}{\tau}}\gamma^{\frac{1}{\tau}}\epsilon^{-\frac{\sigma}{\tau}}$ and we get $N_1>N_2\gg N$.
            \end{proof}
    By the Claim \ref{claim}, we have
            \begin{align}
             Mg(\theta)M^{-1}=&
            \begin{pmatrix}
                \mathrm{i} \widehat{g_{11}}(0)&\widehat{g_{12}}(n_*)e^{\mathrm{i}2\pi\langle n_*,\theta\rangle}\\ \overline{\widehat{g_{12}}(n_*)}e^{-\mathrm{i}2\pi\langle n_*,\theta\rangle}&-\mathrm{i}\widehat{g_{11}}(0)
            \end{pmatrix}
            +\mathcal{R}_g(\theta).
            \end{align}
    Hence we have 
            \[|\mathcal{R}_g|_{r_+}\leq 2 \tilde{\epsilon}e^{-(r-r_+)(2\pi N_2)^{\nu}}\leq 2\epsilon^{\frac{8}{9}} e^{-(r-r_+)(2\pi N_2)^{\nu}}.\]

    Now we define a $(2{\Z})^d$-periodic rotation 
        \[ 
            Z(\theta):=e^{\pi{\langle n_*,\theta\rangle}J}=M^{-1}
            \begin{pmatrix}
                e^{\frac{\mathrm{i}2\pi\langle n_*,\theta\rangle}{2}}&0\\0&e^{-\frac{\mathrm{i}2\pi\langle n_*,\theta\rangle}{2}}
            \end{pmatrix}M, 
        \]   
    where $ J $ is the symplectic matrix $ 
    \begin{pmatrix}
        0&1\\
        -1&0
    \end{pmatrix}
    $.    Obviously, $ Z\in G^{\nu}_{r}(2\mathbb{T}^d, \mathrm{SL}(2,\R)) $, and for any $ r''\in(0,r) $, $ |Z|_{r''}\leq 2e^{r''(\pi |n_*|)^{\nu}} $. Given any $ g\in G_{r}^{\nu\ (re)}(\eta) $, we have \[ Z(\theta+\alpha)^{-1}(\tilde{ A}e^g)Z(\theta)=M^{-1}\begin{pmatrix}
    e^{\mathrm{i}2\pi(\xi-\frac{\langle n_*,\alpha\rangle}{2})}&0\\0&e^{-\mathrm{i}2\pi(\xi-\frac{\langle n_*,\alpha\rangle}{2})}
    \end{pmatrix}Me^{Z^{-1}(\theta)gZ(\theta)}. \]
        By a direct calculation, we get 
        \[ Z(\theta)^{-1}g(\theta)Z(\theta)= M^{-1}\begin{pmatrix}
        \mathrm{i}\widehat{g_{11}}(0)&\widehat{g_{12}}(n_*)\\\overline{\widehat{g_{12}}(n_*)}&-\mathrm{i}\widehat{g_{11}}(0)
        \end{pmatrix}M+Z(\theta)^{-1}\mathcal{R}_g(\theta)Z(\theta):=\mathfrak{p}+\mathfrak{q}. \]
        Let $e^{\mathfrak{p}+\mathfrak{q}}=e^{\mathfrak{p}}e^{f_+}$, $A_+=e^{2\pi \left( \xi-\frac{\langle n_*,\theta\rangle}{2}\right)J}e^\mathfrak{p}:=e^{a_+}$, and $ B:=P\cdot e^Y\cdot Z $, the system $ (\alpha,Ae^f) $ is conjugate to $ (\alpha,A_+e^{f_+}) $ by $ B $.
             
    Now we get the corresponding estimates. By the decay property of Fourier coefficients of $ g(\theta)$, for any $ n\in\Z^d $, we have
            \begin{equation*}\label{fdecay}
            |\widehat{g_{11}}(n)|,|\widehat{g_{12}}(n)|\leq 2\tilde{\epsilon} e^{- r(2\pi|n|)^\nu}.
            \end{equation*} 
            Then by the BCH Formula (\ref{BakerF}), one has
            \begin{align*}
                \|a_+\|&\leq 2\left(\epsilon^{\sigma}+4\epsilon^{\frac{8}{9}}(1+e^{- r(2\pi|n_*|)^\nu})\right)\leq 4\epsilon^{\sigma},\\
                |f_+|_{r_+}&\leq 2\cdot 4e^{ 2r_+(\pi|n_*|)^\nu}\cdot 2\epsilon^{\frac{8}{9}} e^{-(r-r_+) (2\pi N_2)^\nu}\\
                &\quad\qquad\leq \epsilon^{\frac{3}{4}-\frac{2 r_+}{r-r_+}} e^{-(r-r_+) \epsilon^{-\frac{10\nu }{11}\frac{\sigma}{\tau}}}\ll\epsilon^{2},\\
                |B|_{r''}&\leq 32 \gamma^{-\frac{1}{2}}\|A\|^{\frac{1}{2}} |n_*|^{\frac{\tau}{2}}e^{r''(\pi |n_*|)^{\nu}},\quad\forall r''\in(0, r_+].
            \end{align*}
         It is easy to see that $ \deg B=n_* $, and (\ref{3.4}), thus we finish the proof of Proposition \ref{prop}.   
    \end{proof}
\subsection{KAM scheme}
Now we consider the initial cocycle $ (\alpha,A_0e^{f_0}) $. Denote its rotation number by $ \rho(\alpha,A_0e^{f_0}) $, Eliasson \cite{eliasson} proved that if $ \rho(\alpha,A_0e^{f_0}) $ is \textit{Diophantine} with respect to $ \alpha $, that is, there exist $ \kappa^{-1}>1 $, $ \tau>d $ such that $ \rho\in\mathrm{DC}_\alpha^d(\kappa,\tau) $, where
$$
\mathrm{DC}_\alpha^d(\kappa,\tau)=\{\phi\in\R^d: |2\phi-\langle m,\alpha\rangle|_{\T}\geq \frac{\kappa}{(1+|m|)^\tau} \}, 
$$ 
or \textit{rational} with respect to $ \alpha $ (i.e., $ 2\rho(\alpha,A_0e^{f_0})\equiv\langle k, \alpha\rangle\mod\Z$ for some $ k\in \Z^{d} \backslash \{0\}$), then the system $ (\alpha,A_0e^{f_0}) $ is reducible, thus is full measure reducible, for small analytic potentials. Here we will focus on a quantitative version of full measure reducibility. 
\begin{theorem}\label{thm4.1}
	Given $ r_0>0 $. Assume $ \alpha \in \mathrm{DC}_d(\gamma,\tau)$, $ A_0\in\mathrm{SL}(2,\R) $, $f_0\in  G^\nu_{r_0}(\mathbb{T}^d,\mathrm{sl(2,\R)})$, with $0<\nu<1$. Then for any $ r\in(0,r_0) $, there exists $\bar{c}=\bar{c}(\gamma,\tau,d,\nu)>0 $ and a numerical constant $ \bar{D} $, such that if 
    \begin{equation}\label{4.9}
        |f_0|_{r_0}=\epsilon_0<\epsilon_*:= \bar{c}\lVert A_0 \rVert^{-\bar{D}}(r_0-r)^{\bar{D}\nu\tau}, 
    \end{equation}
    the following holds:
    \begin{enumerate}[label=$\mathrm{(\roman*)}$\ \ ,align=left,widest=ii,labelsep=0pt,leftmargin=0.5em]
		\item The system $ (\alpha,A_0e^{f_0}) $ is strong almost reducible with the ``width'' $r$.
		\item If $ \rho(\alpha,A_0e^{f_0}) $ is rational w.r.t. $\alpha $, then there exist $ \widetilde{Z}\in G^\nu_{r}(2\mathbb{T}^d,\mathrm{SL}(2,\R)) $ and $ A\in\mathrm{SL}(2,\R) $ such that 
        \begin{equation} 
        \widetilde{Z}(\theta+\alpha)^{-1}A_0e^{f_0(\theta)}\widetilde{Z}(\theta)=A,
        \end{equation} 
        with estimates 
        $$
        |\widetilde{Z}|_{r''}\leq D_1|k|^{\tau}e^{2r''(\pi|k|)^{{\nu}}} , \forall r''\in (0,r],\ D_1=D_1(\gamma,\tau,\lVert A_0\rVert,r_0,d)>0.
        $$
		Moreover, if $ (\alpha,A_0e^{f_0}) $ is not uniformly hyperbolic,
		then $ A $ can be chosen as
        $$
        \begin{pmatrix}
            1 &\varphi \\
            0 & 1
            \end{pmatrix} \text{ with } |\varphi|\leq \epsilon_{0}^{\frac{3}{5}} e^{- r|2\pi k|^{{\nu}}}.
        $$

        \item If $ \rho(\alpha,A_0e^{f_0})\in \bigcup_{\gamma'>0}\mathrm{DC}_\alpha^d(\gamma',\tau) $, 
        then there exist $ B\in G^\nu_r(2\T^d,\mathrm{SL}(2,\R)) $, $ A\in\mathrm{SL}(2,\R) $ such that
        \begin{equation}\label{rotationreducible}
            B(\theta+\alpha)^{-1}A_0e^{f_0(\theta)}B(\theta)=A, \quad \rho(\alpha,A)\neq 0,
        \end{equation} 
        with estimates $ |B|_{r}\leq \Gamma_1(\gamma',\gamma,\tau,d,\nu,r,r_0,\lVert A_0\rVert) $, $ |\deg B|\leq \Gamma_2(\gamma',\gamma,\tau,d,\nu,r,r_0,\lVert A_0\rVert) $.
	\end{enumerate}
\end{theorem}

\begin{proof}
	We prove (i) by induction.
    
    1). Take $r_0,r,\epsilon_*$ such that the first step holds.
    
    2). Assume that we are at the $ (j+1)^{\mathrm{th}} $ KAM step, where we have $ A_j\in \mathrm{SL}(2,\R) $ with eigen-values $ e^{\pm\mathrm{i}2\pi \xi_j} $ and $ f_j\in G^\nu_{r_j} $ is satisfied $ |f_j|_{r_j}\leq \epsilon_j $ for some $ \epsilon_j\leq\epsilon_0 $. Let $ \tilde{r} =\frac{r_0+r}{2}$. Then we define
	\begin{alignat}{2}\label{24}
	r_j-r_{j+1}:=\frac{r_0-\tilde{r}}{4^{j+1}}, &\qquad&N_j:=\frac{1}{2\pi}\left(\frac{|\ln \epsilon_j|}{r_j-r_{j+1}}\right)^{\frac{1}{\nu}}=\frac{1}{2\pi}\left(\frac{4^{j+1}|\ln \epsilon_j|}{r_0-\tilde{r}}\right)^{\frac{1}{\nu}}
	.
	\end{alignat}
	If $ \epsilon_j $ is sufficiently small such that $\epsilon=\epsilon_j$, $r=r_j$, $r_+=r_{j+1}$, $A=A_j$ satisfying (\ref{epsilon}), then by the Proposition \ref{prop}, there exist
	\begin{alignat*}{3}
	Z_j\in G^\nu_{r_{j+1}}(2\mathbb{T}^d,\mathrm{SL}(2,\R)), &\quad&A_{j+1}\in \mathrm{sl}(2,\R), &\quad&f_{j+1} \in G^\nu_{r_{j+1}}(\mathbb{T}^d,\mathrm{sl}(2,\R)),
	\end{alignat*}
	such that 
    \begin{equation} \label{almostreduciblej}
        Z_j(\theta+\alpha)^{-1}A_je^{f_j(\theta)}Z_j(\theta)=A_{j+1}e^{f_{j+1}(\theta)}.
    \end{equation} 
    Moreover,
	\begin{enumerate}[label=$\arabic*^\circ$,align=left,widest=ii,labelsep=10pt,leftmargin=1em]
		\item  If for any $ n\in\Z^d , 0<|n|\leq N_j$, we have $ |2\xi_j-\langle n,\alpha\rangle|_{\T}\geq \epsilon^{\frac{1}{20}} $, then
		\begin{alignat}{3}\label{25}
		\|A_{j+1}-A_j\|\leq 4\|A_j\|\epsilon_j, &\quad&|f_{j+1}|_{r_{j+1}}\leq \epsilon_{j+1}:=\epsilon_j^2,&\quad& |Z_j-\mathrm{Id}|_{r_{j+1}}\leq 2\epsilon_j^{\frac{1}{2}}. 
		\end{alignat}
		\item If there is $ n_j\in \Z^d $, $ 0<|n_j|\leq N_j $ such that $ |2\xi_j-\langle n_j,\alpha\rangle|_{\T}\leq \epsilon^{\frac{1}{20}}, $ then 
        \begin{equation} \label{structure}
            A_{j+1}=e^{a_{j+1}}=M^{-1}\exp \begin{pmatrix}
                \mathrm{i}t_j & \mu_j\\
                \overline{\mu_j} & -\mathrm{i}t_j    
            \end{pmatrix}M 
        \end{equation}  
        with estimates:
		\begin{equation}\label{26}
		\begin{aligned}
            \|a_{j+1}\|&\leq  4\epsilon_j^{\frac{1}{20}},\quad|t_j|\leq \epsilon_j^{\frac{1}{20}},\quad|\mu_j|\leq \epsilon_j^{\frac{3}{4}}e^{-r|2\pi n_*|^{\nu}}, \\ 
         |f_{j+1}|_{r_{j+1}}&\leq\epsilon_{j+1}:=\epsilon_{j}^{\frac{3}{4}-\frac{2 r_{j+1}}{r_j-r_{j+1}}} e^{-(r_j-r_{j+1}) \epsilon_j^{-\frac{1}{22}\frac{\nu}{\tau}}}. 
        \end{aligned}
		\end{equation}
		and
		\begin{equation}\label{27}
		|Z_j|_{r''} \leq 32\gamma^{-\frac{1}{2}}\|A_j\|^{\frac{1}{2}}|n_j|^{\frac{\tau}{2}} e^{r''(\pi|n_j|)^\nu}
		, \quad\forall r''\in (0,r_{j+1}].
		\end{equation} 
	\end{enumerate}

	3). In view of (\ref{25}) and (\ref{26}) one sees that for any $ j\geq0 $, $ \epsilon_j\leq\epsilon_0^{2^j} $, $ \|A_j\|\leq2\|A_0\| $. So if $ \epsilon_0<\epsilon_* $, then Proposition \ref{prop} can be applied iteratively. Indeed, it suffices to consider (\ref{epsilon}) in each KAM step. $\epsilon_j$ decays super-exponentially with $j$, while the right hand side of (\ref{epsilon}) decays exponentially with $j$.
	Hence $ (\alpha, A_0e^{f_0}) $ is almost reducible. 
	
	For (ii). Assume that there are at least two resonant steps in the above procedure. We consider two consecutive resonant steps, assume they are the $ p^{\mathrm{th}} $ and the $ q^{\mathrm{th}} $. If at the $ q^{\mathrm{th}} $, by the resonant condition, $ \left\lvert 2\xi_{ q} -\langle n_{q},\alpha\rangle\right\rvert_{\T}\leq \epsilon_{q}^{\frac{1}{20}} $, hence $ \lvert\xi_{ q}\rvert>\frac{\gamma}{4|n_{q}|^\tau} $. By the Proposition \ref{prop}, after the $ p^{\mathrm{th}} $, we have $ \lvert\xi_{ p+1}\rvert \leq 4\epsilon_{p}^{\frac{1}{20}}$. By (\ref{25}) and (\ref{4.9}), we have $ |\xi_{q} |\leq 8\epsilon_{p}^{\frac{1}{20}}\leq \frac{\epsilon_{p}^{\frac{1}{22}}\gamma}{4|n_{p}|^\tau}$. Thus, 
    \begin{equation}\label{28}
	|n_{q}|\geq \epsilon_{p}^{-\frac{1}{22\tau}}|n_{p}|.
	\end{equation}
	If the $ (j_m)^{\mathrm{th}} $ is the resonant step, then $ \deg Z_{j_m}=n_{j_m} $. By (\ref{degree}) and (\ref{28}) and $ 2\rho(\alpha,A_0e^{f_0})\equiv\langle k, \alpha\rangle \mod\Z$, we have $ \frac{\langle k,\alpha\rangle}{2}=\rho(\alpha,A_{j_m+1}e^{f_{{j_m+1}}}) +\sum_{l=1}^m \frac{\langle n_{j_l},\alpha\rangle}{2},  $ then 
    \begin{align*}
	8\pi\epsilon_{j_m}^{\frac{1}{20}}\geq|2\pi\xi_{j_m+1}|&=|\rho(\alpha, A_{j_m+1}e^{f_{{j_m+1}}})|=\frac{1}{2}|\langle k-\sum_{l=1}^{m} n_{j_l},\alpha\rangle|\gtrsim \frac{1}{N_{j_m}^\tau}, 
	\end{align*}
	thus there are at most finitely many resonant steps.
	
	Then by the estimates of $ Z_{j} $ in (\ref{25}) and the sequence $ (r_j)_{j\in\mathbb{N}} $ given in (\ref{24}), we can find that $ \prod_{j=0}^{\infty} Z_{j} = \widetilde{Z}\in G^\nu_{\tilde{r}}(2\mathbb{T}^d, \mathrm{SL}(2,\R)) $ is well-defined, thus we have $ \widetilde{Z}(\theta+\alpha)^{-1}(A_0e^{f_0})\widetilde{Z}=e^a $ for some $ a\in \mathrm{sl}(2,\R) $ with $ \rho(\alpha,e^a)=0 $. 
	
	Assume that there are $ s+1 $ resonant steps totally, and the resonant points are\[ n_{j_0},\dots,n_{j_s} \in\Z^d,\quad 0<|n_{j_i}|\leq N_{j_i}, \, i=0,1,\dots,s,\] then $ k= n_{j_0}+\cdots+n_{j_s}  $. By (\ref{28}), we have 
    $$ 
    \begin{aligned}
    |n_{j_s}|-\sum_{i=0}^{s-1}|n_{j_i}|&\leq |k|\leq |n_{j_s}|+\sum_{i=0}^{s-1}|n_{j_i}|,\\
    (1-2\epsilon_{j_{s-1}}^{\frac{1}{22\tau}})^\nu |n_{j_s}|^{\nu}&\leq|k|^{\nu}\leq(1+2\epsilon_{j_{s-1}}^{\frac{1}{22\tau}})^{\nu} |n_{j_s}|^{\nu},\\
    \prod_{i=0}^s \text{const.}|n_{j_i}|&\leq (\text{const.}|n_{j_s}|)^2.
    \end{aligned}
    $$ 
    Then by (\ref{27}), for any $r''\in(0,r],$
\begin{align*}
|\widetilde{Z}|_{r''}&\leq 4|Z_{j_0}|_{r''}\cdots|Z_{j_s}|_{r''}\\
&\leq 4\prod_{i=0}^s 32\gamma^{-\frac{1}{2}}\|A_{j_i}\|^{\frac{1}{2}}|n_{j_i}|^{\frac{\tau}{2}} e^{r''(\pi|n_{j_i}|)^{{\nu}}}\leq  D_1 |k|^{\tau}e^{2r''(\pi|k|)^{{\nu}}},
\end{align*}
where $D_1=D_1(\gamma,\tau,\|A_0\|,r_0,d)>0$. 

If $ (\alpha,A_0e^{f_0}) $ is not uniformly hyperbolic, which means $ a $ can not be a hyperbolic matrix, thus we have $ \det a=0. $ Assume $ a=\begin{pmatrix}
	a_{11}&a_{12}\\a_{21}&-a_{11}
	\end{pmatrix} $. Then exists $ \vartheta \in\mathbb{T} $ such that $ R_{-\vartheta}aR_\vartheta=\widetilde{a}:=\begin{pmatrix}
	0&a_{12}-a_{21}\\0&0
	\end{pmatrix} $, replace $ \widetilde{Z} $ by $ \widetilde{Z}R_\vartheta $, and let $ \varphi=a_{12}-a_{21} $, we get $ (\alpha,A_0e^{f_0}) $ is conjugated to $ (\alpha,e^{\widetilde{a}}) $. 
	
	Noting that if we rewrite $a$ as $a=M^{-1}\begin{pmatrix}
	\mathrm{i}t&\mu\\\overline{\mu}&-\mathrm{i}t
	\end{pmatrix}M$ with $t\in\R$, $\mu\in\mathbb{C}$, by (\ref{isom}), we have $|\varphi|=2|t|$. Now we consider after the last resonant step $ (\alpha,A_{j_s+1}e^{f_{j_s+1}}) $: 
    
    Clearly, we have 
$$e^a=e^{a_{j_s+1}}\prod_{l=1}^{\infty}e^{\widehat{g_l}(0)},$$
then apply the BCH Formula (\ref{BakerF}) again, we have 
	\begin{equation*}
	\|a_{j_s+1}-a \|\leq 2\|\widehat{g_1}(0)\|\leq 4\epsilon_{j_s+1}.
	\end{equation*}
	
	By the structure of $a_{j_s+1}$ (c.f. (\ref{structure})), and $\det a=0$, we have
	\[|t|=|\mu|\leq\epsilon_{j_s}^{\frac{3}{4}} e^{- r_{j_s}(2\pi |n_{j_s}|)^\nu}+4\epsilon_{j_s+1}.\]
Since $ \epsilon_{j_s+1}:=\epsilon_{j_s}^{\frac{3}{4}-\frac{2r_{j_s+1}}{r_{j_s}-r_{j_s+1}}} e^{-(r_{j_s}-r_{j_s+1}) \epsilon_{j_s}^{-\frac{10\nu}{11}\frac{\sigma}{\tau}}}$, we get	
\begin{equation} \label{3.15}
    |\varphi|\leq 5\epsilon_{j_s}^{\frac{3}{4}} e^{- r_{j_s}(2\pi|n_{j_s}|)^{{\nu}}}\leq\epsilon_{0}^{\frac{3}{5}}e^{-r|2\pi k|^{{\nu}}}.
\end{equation} 
For (iii). Assume that we are at the $ (j+1)^{\mathrm{th}} $ KAM step, that is, we have (\ref{almostreduciblej}). Then set $ B_0=\mathrm{Id} $, and let $ B_{j+1}=B_jZ_j(\theta) $, we have 
$$
    B_{j+1}(\theta+\alpha)^{-1}A_0e^{f_0(\theta)}B_{j+1}(\theta)=A_{j+1}e^{f_{j+1}(\theta)},
$$ 
we distinguish two cases:\\ 
\textrm{Non-resonant case:} By $ 1^\circ $, , we have 
$$
    |B_{j+1}|_{r_{j+1}}\leq (1+2\epsilon_j^{\frac{1}{2}})|B_{j}|_{r_j}, \text{ and } \deg B_{j+1}=\deg B_j.
$$ 
\textrm{Resonant case:} By $ 2^\circ $, review Proposition \ref{prop}, we have 
$$
\begin{aligned}
|B_{j+1}|_{r_{j+1}}&\leq 32\gamma^{-\frac{1}{2}}\|A_j\|^{\frac{1}{2}}|n_j|^{\frac{\tau}{2}}e^{r_{j+1}(\pi|n_{j_i}|)^{{\nu}}}|B_{j}|_{r_j}, \text{ and } \deg B_{j+1}=\deg B_j+n_j.
\end{aligned}
$$ 
Thus we have
$$
    \begin{aligned}
        |\deg B_j|&\leq \sum_{i=0}^{j-1} N_i\leq \sum_{i=0}^{j-1}\frac{1}{2\pi}\left(\frac{4^{j+1}|\ln \epsilon_j|}{r_0-\tilde{r}}\right)^{\frac{1}{\nu}}\leq \frac{|\ln\epsilon_j|^3}{r_0-\tilde{r}}.
    \end{aligned}
$$ 
Since $ \rho(\alpha,A_0e^{f_0(\theta)})\in \mathrm{DC}_{\alpha}^d(\gamma',\tau) $ with some $ \gamma'>0 $, for any $ m\in\Z^d $, we have
$$
    \begin{aligned}
        |2\rho(\alpha,A_je^{f_j(\theta)})-\langle m,\alpha\rangle|_{\T}&=|2\rho(\alpha,A_0e^{f_0(\theta)})+\langle \deg B_j,\alpha\rangle-\langle m,\alpha\rangle|_{\T}\\
        &\geq \frac{\gamma'}{(|m-\deg B_j|+1)^\tau}\geq \frac{\gamma'(1+|\deg B_j|)^{-\tau}}{(|m|+1)^\tau}
    \end{aligned}
$$ 
which means $ \rho(\alpha,A_je^{f_j(\theta)})\in\mathrm{DC}_\alpha^d(\gamma'(1+|\deg B_j|)^{-\tau},\tau) $.

Then we will need the following lemma which is essentially proved by Dinaburg-Sinai \cite{dinaburg1975one}, see also \cite{ge2019exponential}.
\begin{lemma}\label{DinaburgSinai}
    Let $ \rho(\alpha,A_je^{f_j(\theta)})\in\mathrm{DC}_\alpha^d(\kappa,\tau) $, then if 
    \begin{equation} \label{lem4.3}
        |f_j|_{r_j}\leq \epsilon_j< c\kappa^{40\nu}\lVert A \rVert^{-D}(r_j-r_{j+1})^{D\nu\tau}, 
    \end{equation} 
    there exist $ \bar{Z}\in G^\nu_{r}(\T^d,\mathrm{SL}(2,\R)) $, $ A\in\mathrm{SL}(2,\R) $ such that
    $$
    \bar{Z}(\theta+\alpha)^{-1}A_je^{f_j(\theta)}\bar{Z}(\theta)=A,
    $$ 
    with estimates $ |\bar{Z}-\mathrm{Id}|_{r}\leq 4\epsilon_j^{\frac{1}{2}} $, $ \lVert A-A_j\rVert\leq 8\lVert A_0\rVert\epsilon_j $.
\end{lemma}
\begin{proof}
    By (\ref{lem4.3}), (\ref{degree}) and $ \rho(\alpha,A_je^{f_j(\theta)})\in \mathrm{DC}_\alpha^d(\kappa,\tau) $, for any $ 0<|n|\leq N_j $, we have 
$$
    \begin{aligned}
        |2\pi\xi_j-\langle n,\alpha\rangle|_\T&\geq |\rho(\alpha,A_je^{f_j(\theta)})-\langle n,\alpha\rangle|_{\T}-|2\pi\xi_j-\rho(\alpha,A_je^{f_j(\theta)})|\\
        &\geq \frac{\kappa}{(1+|n|)^{\tau}}-\epsilon_j\geq \frac{\kappa}{(1+N_j)^{\tau}}-\epsilon_j \geq \epsilon_j^{\frac{1}{20}}.
    \end{aligned}
$$ 
Thus there are no resonant steps, then by (\ref{25}), there exists $ \bar{Z}(\theta):=\prod_{j=0}^{\infty}Z_j(\theta)\in G_r^\nu(\T^d,\mathrm{SL}(2,\R)) $ such that 
$$
\bar{Z}(\theta+\alpha)^{-1}A_je^{f_j(\theta)}\bar{Z}(\theta)=A,
$$ 
with $ | \bar{Z}-\mathrm{Id}|_{r}\leq 4\epsilon_j^{\frac{1}{2}} $, $ \lVert A-A_j\rVert\leq 8\lVert A_0\rVert\epsilon_j $.
\end{proof}
Notice that $ \epsilon_j\leq e_0^{2^j} $, one can select $ j_0\in\Z $ to be the smallest integer such that
$$
    \epsilon_j\leq c(\frac{\gamma'}{(1+|\deg B_j|)^{\tau}})^{40\nu}\lVert A \rVert^{-D}(r_j-r_{j+1})^{D\nu\tau},
$$ 
then by Lemma \ref{DinaburgSinai}, there exists $ \bar{Z}\in G_r^\nu(\T^d,\mathrm{SL}(2,\R)) $ such that
$$
    \bar{Z}(\theta+\alpha)^{-1}B_{j_0}(\theta+\alpha)^{-1}A_0e^{f_0(\theta)}B_{j_0}(\theta)\bar{Z}(\theta)=A.
$$ 
Let $ B=B_{j_0}\bar{Z}(\theta) $, by (\ref{28}), any two consecutive resonant steps separate exponentially, we have
$$
|B|_r\leq 2|B_{j_0}|_{r_{j_0}}\leq D_1 |n_{j_0}|^\tau e^{2r_0|\pi n_{j_0}|^\nu}\leq \Gamma_1(\gamma',\gamma,\tau,d,\nu,r,r_0,\lVert A_0\rVert),
$$
$$ 
|\deg B|=|\deg B_{j_0}|\leq \frac{|\ln\epsilon_{j_0}|^3}{r_0-\tilde{r}}\leq\Gamma_2(\gamma',\gamma,\tau,d,\nu,r,r_0,\lVert A_0\rVert).
$$
Notice that $ \rho(\alpha,A)\neq 0 $, otherwise it will contradict to $ \rho(\alpha,A_0e^{f_0(\theta)})\in\mathrm{DC}_\alpha^d(\gamma',\tau) $. 
\end{proof}
\section{Proof of Theorem \ref{thmlong}}
A solution $ u_E^\phi(n) $ of $ L_{\lambda^{-1}v,\alpha,\phi}u_E^\phi(n)=E u_E^\phi(n) $ is called a \textit{normalized eigenfunction} of $ L_{\lambda^{-1}v,\alpha,\phi} $ if $ \sum_n |u_E^\phi(n)|^2=1 $. 
\begin{Definition}
    For any fixed $ N\in\N $, $ C>0 $, $ \varepsilon>0 $, a normalized eigenfunction $ u_E^{\phi}(n) $ of $ L_{\lambda^{-1}v,\alpha,\phi} $ is said to be $ (\nu,N,C,\varepsilon)$-good if $ |u_E^{\phi}(n)|\leq e^{-C\varepsilon |n|^{\nu}} $ for $ |n|\geq (1-\varepsilon)N $.
\end{Definition}
We denote by $ u_j^\phi(n) $ the $ (\nu,N,C,\varepsilon)$-good eigenfunctions of $ L_{\lambda^{-1}v,\alpha,\phi} $ and by $ E_j^\phi $ the corresponding eigenvalues. Let 
$$
    \mathcal{E}_{\nu,N,C,\varepsilon}^\phi=\{E_j^\phi:u_j^\phi(n) \text{ is a } (\nu,N,C,\varepsilon)\text{-good eigenfunction of } L_{\lambda^{-1}v,\alpha,\phi} \}.
$$  
Then we have the following criterion.
\begin{proposition}
    For a.e. $ \phi\in\T $, suppose that there exists $ C>0 $, and for any $ \delta>0 $, there exist $ N_0(\delta)>0 $ and $ 0<\varepsilon(\delta,N_0,C,\nu)<1 $, such that
    \begin{equation}\label{sharp1}
        \sharp\mathcal{E}_{\nu,N,C,\varepsilon}^\phi\geq (1-\delta)(2N)^d
    \end{equation} 
    holds for any $ N\geq N_0(\delta) $ and $ \varepsilon=\varepsilon(\delta,N_0,C,\nu) $. Then $ L_{\lambda^{-1}v,\alpha,\phi} $ has pure point spectrum with sub-exponentially for a.e. $ \phi\in\T $.
\end{proposition}
\begin{proof}
    The proof are essentially contained in \cite{avila2017sharp}, we give the proof here for completeness.

    Let $ P^\phi(\bigtriangleup) $ be the spectral projection of $ L_{\lambda^{-1}v,\alpha,\phi} $ on to the set $ \bigtriangleup $, then for any $ \delta_n\in\ell^2(\Z) $, it induces the spectral measure:
    $$
        \langle P_{\bigtriangleup}^\phi \delta_n,\delta_n\rangle=\mu_{\delta_n,\phi}(\bigtriangleup).
    $$ 
    Now we fix $ \phi\in\T $ such that (\ref{sharp1}) is satisfied. Denote that
    $$
        K_{\nu,N,C,\varepsilon}^\phi=\{j\in\N:u_j^\phi(n) \text{ is a } (\nu,N,C,\varepsilon)\text{-good eigenfunction of } L_{\lambda^{-1}v,\alpha,\phi}\}.
    $$ 
    Notice that $ \sharp K_{\nu,N,C,\varepsilon}^\phi $ is finite for any fixed $ \nu,N,C,\varepsilon $, and pick $ N\geq\max\{N_0(\delta),\frac{|n_0|}{\varepsilon}\} $, one has
    $$
        \sum_{|n-n_0|\leq N}|u_j^\phi(n)|^2>\sum_{|n|\leq N(1-\varepsilon)}|u_j^\phi(n)|^2>1-\mathcal{C}(C,\varepsilon,\nu,d)2^d N^{d-1}e^{-C\varepsilon N^\nu(1-\varepsilon)^\nu},
    $$ 
    for $ (\nu,N,C,\varepsilon)$-good eigenfunction $ u_j^\phi(n) $.

    Let $ \mathcal{E}_{C}(\phi)=\bigcup_{\varepsilon>0}\bigcup_{N>0}\mathcal{E}_{N,C,\varepsilon}^\phi $. We have
    $$
        \begin{aligned}
            \frac{1}{(2N)^d}\sum_{|n-n_0|\leq N}|\mu_{\delta_n,\phi}(\mathcal{E}_{C}(\phi))|&> \frac{1}{(2N)^d}\sum_{|n|\leq N(1-\varepsilon)}|\langle P_{\mathcal{E}_{N,C,\varepsilon}^\phi}^\phi \delta_n,\delta_n\rangle|\\
            &=\frac{1}{(2N)^d}\sum_{|n|\leq N(1-\varepsilon)}\sum_{j\in K_{\nu,N,C,\varepsilon}^\phi }|u_j^\phi(n)|^2\\
            &>\frac{\sharp K_{\nu,N,C,\varepsilon}^\phi }{(2N)^d}(1-\mathcal{C}(C,\varepsilon,\nu,d)2^d N^{d-1}e^{-C\varepsilon N^\nu(1-\varepsilon)^\nu})\\
            &>(1-\delta)(1-\mathcal{C}(C,\varepsilon,\nu,d)2^d N^{d-1}e^{-C\varepsilon N^\nu(1-\varepsilon)^\nu}).
        \end{aligned}
    $$ 
    Since $ \mathcal{E}_{C}(\phi)=\mathcal{E}_{C}(\phi+\langle n,\alpha\rangle) $, and notice that $ \langle P_{\bigtriangleup}^\phi \delta_n,\delta_n\rangle=\langle P_{\bigtriangleup}^{\phi+\langle m,\alpha\rangle} \delta_{n+m},\delta_{n+m}\rangle $, we can rewrite the inequalities above as
    $$
        \frac{1}{(2N)^d}\sum_{|n|\leq N}|\mu_{\delta_{n_0},\phi+\langle n,\alpha\rangle}(\mathcal{E}_{C}(\phi))|\geq (1-\delta)(1-\mathcal{C}(C,\varepsilon,\nu,d)2^d N^{d-1}e^{-C\varepsilon N^\nu(1-\varepsilon)^\nu}).
    $$ 
    Let $ N\to\infty $, and then let $ \delta\to 0 $, we have
    $$
        \int_{\T}|\mu_{\delta_{n_0},\phi}(\mathcal{E}_{C}(\phi))|\dif \phi=1
    $$ 
    by the ergodic theorem of $ \Z^d $ actions \cite{schmidt1996ergodic}. It follows that $ |\mu_{\delta_{n_0},\phi}(\mathcal{E}_{C}(\phi))|=1 $ for a.e. $ \phi $, for arbitrary fixed $ n_0\in\Z^d $. Note that $ u_j^\phi(n) $ decays sub-exponentially if $ E_j^\phi \in \mathcal{E}_{C}(\phi) $, which completes the proof.
\end{proof}
Denote $ \Theta_\gamma=\{\phi: \phi\in\mathrm{DC}_\alpha^d(\gamma,\tau)\} $. Note that $ |\bigcup_{\gamma>0}\Theta_\gamma|=1 $ implies that for any $ \delta>0 $, there exists $ \epsilon_0 $ such that if $ |\gamma'|\leq\epsilon_0 $, then $ |\Theta_{\gamma'}|\geq 1-\frac{\delta}{3} $. By ergodicity of rigid irrational rotations, we have 
$$
    \lim_{N\to\infty}\frac{1}{(2N)^d}\sum_{|n|\leq N}\chi_{\Theta_{\gamma'}}(\phi+\langle n,\alpha\rangle)=\int_{\T}\chi_{\Theta_{\gamma'}}(\phi)\dif \phi=|\Theta_{\gamma'}|\geq 1-\frac{\delta}{3}.
$$  
Thus we have for a.e. $ \phi\in\T $,
$$
    \sharp\{n:\phi+\langle n,\alpha\rangle\in\Theta_{\gamma'},|n|\leq N(1-\frac{\delta}{3})\}\geq (1-\delta)(2N)^d
$$ 
holds for $ N\geq N_0(\delta) $.

Then we show that for any $ N\geq N_0(\delta) $, any $ \rho(\alpha,S_{E_m}^{\lambda^{-1}v})=\phi+\langle m,\alpha\rangle \in\mathrm{DC}_\alpha^d(\gamma',\tau) $ with $ |m|\leq N(1-\frac{\delta}{3}) $ can produce $ (\nu,N,C,\varepsilon) $-good eigenfunctions. Consequently, we have (\ref{sharp1}) holds, then we can conclude Theorem \ref{thmlong}.

Suppose that $ \rho(\alpha,S_{E_m}^{\lambda^{-1}v})=\phi+\langle m,\alpha\rangle \in\mathrm{DC}_\alpha^d(\gamma',\tau)$, then by Theorem \ref{thm4.1}, for any $ r\in(0,r_0) $, when $ \lambda>\lambda_0(\alpha,d,\nu,r_0,r)$, there exists $ B_m\in G_r^\nu(2\T^d,\mathrm{SL}(2,\C)) $ (by adding a unitary) such that
$$
    B_m(\theta+\alpha)^{-1}S_{E_m}^{\lambda^{-1}v}B_m(\theta)=\begin{pmatrix}
        e^{\mathrm{i}2\pi(\phi+\langle m',\alpha\rangle)}&c\\
        0&e^{-\mathrm{i}2\pi(\phi+\langle m',\alpha\rangle)}
    \end{pmatrix},
$$ 
with estimates (omit other dependencies)
$$
    |B_m|_r\leq C_1(\gamma'), \quad|m-m'|\leq 2|\deg B_m|\leq C_2(\gamma').
$$ 

Let $ B_m(\theta)=\begin{pmatrix}
    b_{11}(\theta)& b_{12}(\theta)\\
    b_{21}(\theta)& b_{22}(\theta)
\end{pmatrix} $, a direct calculation shows that
$$
    (E-\lambda^{-1}v)b_{11}(\theta)=b_{11}(\theta-\alpha)e^{-\mathrm{i}2\pi(\phi+\langle m',\alpha\rangle)}+b_{11}(\theta+\alpha)e^{\mathrm{i}2\pi(\phi+\langle m',\alpha\rangle)},
$$ 
we denote $ z_{11}(\theta)=e^{\mathrm{i}2\pi\langle m',\theta\rangle}b_{11}(\theta) $, then one has 
\begin{equation} \label{fourier}
    (E-\lambda^{-1}v)z_{11}(\theta)=z_{11}(\theta-\alpha)e^{-\mathrm{i}2\pi\phi}+z_{11}(\theta+\alpha)e^{\mathrm{i}2\pi\phi}.
\end{equation} 
Taking the Fourier transformation of (\ref{fourier}), we have
\begin{equation} 
    \sum_{k\in\Z^d}\hat{z}_{11}(n-k)\hat{v}_k+2\lambda\cos 2\pi(\phi+\langle n,\alpha\rangle)\hat{z}_{11}(n)=\lambda E \hat{z}_{11}(n),
\end{equation} 
i.e., $ \{\hat{z}_{11}(n)\}_{n\in\Z^d} $ is an eigenfunction of the long range operator $ L_{\lambda^{-1}v,\alpha,\phi} $. By the fact that $ \det|B_m(\theta)|=1 $, we have 
$$
    2\lVert b_{11}\rVert_{L^2}=\lVert b_{11}\rVert_{L^2}+\lVert b_{21}\rVert_{L^2}\geq \lVert B_m\rVert_{0}^{-1},
$$ 
which implies $ \lVert \hat{z}_{11}\rVert_{\ell^2}=\lVert \hat{b}_{11}\rVert_{\ell^2}=\lVert b_{11}\rVert_{L^2}\geq (2\lVert B_m\rVert_0^{-1})$. Take $ u_{m}^{\phi}(n)=\frac{\hat{z}_{11}(n)}{\lVert \hat{z}_{11}\rVert_{\ell^2} } $, then it is a normalized eigenfunction of $ L_{\lambda^{-1}v,\alpha,\phi} $. Let 
$$
    2\varepsilon<1-(1-\frac{\delta}{3})^\nu-\frac{2\ln C_1}{r|2\pi|^\nu N^\nu}-\frac{|C_2|^\nu}{N^\nu}.
$$ 
Since $ u_{m}^{\phi}(n)=u_{m}^{\phi+\langle m',\alpha\rangle}(n-m') $ and the subadditivity of $ f(x)=x^\nu $, then we have
\begin{equation*} 
    \begin{aligned}
        |u_{m}^{\phi}(n)|&=|u_{m}^{\phi+\langle m',\alpha\rangle}(n-m')|\\
        &\leq |B_m|_{r}^2 e^{-r|2\pi (n-m')|^\nu}\\
        &\leq C_1^2e^{-r|2\pi|^{\nu}(\sum_{i=1}^d|n_i-m'_i|)^\nu}\\
        &\leq C_1^2e^{-r|2\pi|^{\nu}(\sum_{i=1}^d|n_i|-|m'_i|)^\nu}\\
        &\leq C_1^2e^{-r|2\pi|^{\nu}(|n|^\nu-|m'|^\nu)}\\
        &\leq C_1^2e^{r|2\pi|^{\nu}(|m|+|m-m'|)^\nu}e^{-r|2\pi n|^\nu}\\
        &\leq e^{r|2\pi|^{\nu} (\frac{2\ln C_1}{r|2\pi|^{\nu} }+|m|^{\nu}+ |C_2|^\nu)}e^{-r|2\pi n|^\nu}\\
        &\leq e^{|2\pi|^\nu r\varepsilon  |n|^{\nu}},
    \end{aligned}
\end{equation*} 
for $ |n|\geq N(1-\delta)  $, which means $ \{u_{m}^{\phi}(n)\}_{n\in\Z^d} $ is $ (\nu,N,|2\pi|^\nu r,\varepsilon) $-good. Notice that $ \varepsilon $ can be chosen as $ \varepsilon(\delta,N_0,C,\nu) $.

\section{Proof of Theorem \ref{thm1.2}}
In this section, we prove Theorem \ref{thm1.2} by the famous Moser-P\"oschel argument of quasi-periodic Schr\"odinger cocycles and our reducibility results. 
Rewrite the Schr\"odinger cocycle as \[S_E^v=A_E+F(\theta),\]
where \[A_E=\begin{pmatrix}
E&-1\\1&0
\end{pmatrix}, F(\theta)=\begin{pmatrix}
-v&0\\0&0
\end{pmatrix}.\]
We only need to consider \[E\in\Sigma_{v,\alpha}\subset [-2+\inf v,2+\sup v],\]
since $\R\backslash\Sigma_{v,\alpha}$ is uniformly hyperbolic. Thus the norm of $A_E$ is bounded uniformly with respect to $E$. Noting that 
\[
    A_E+F(\theta)=A_EG=A_E\begin{pmatrix}
1&0\\v&1
\end{pmatrix}, 
\]
and $ \mathrm{tr} G=2 $, by the exponential map, there exists $f\in G^\nu_{r_0}(\mathbb{T}^d,\mathrm{sl}(2,\R))$ such that 
\[
    A_E+F=A_Ee^{f}, \text{ with } |f|_{r_0}\sim|v|_{r_0}.
\]

Fix $r\in(0,r_0)$, we consider the cocycle $ (\alpha, S^v_{E_k^+}) $, where $ E_k^+\in\Sigma_{v,\alpha} $ is the right edge point of gap $G_k(v)$, then we have
$ 2\rho(\alpha,S_{E_k^+}^v)\equiv\langle k, \alpha\rangle\mod\Z$, and $ (\alpha,S_{E_k^+}^v) $ is not uniformly hyperbolic. By Theorem \ref{thm4.1}, there exists $\epsilon_*=\epsilon_*(\gamma,\tau,\nu,r_0,r,d)>0$ such that if $ |v|_{r_0}=\epsilon_0<\epsilon_*$, then there exist $ Z\in G^\nu_{r}(2\mathbb{T}^d,\mathrm{SL}(2,\R)) $, $c\in(0,1)$ ($  c= 0 $ iff it is a collapsed spectral gap, thus we need to do nothing), such that
\[ Z(\cdot+\alpha)^{-1} S^v_{E_k^+}Z(\cdot)=B=\begin{pmatrix}
		1 &c \\
		0 & 1
		\end{pmatrix}.\]

Noting that for any $\delta>0$, we have
\[  
    Z(\cdot+\alpha)^{-1} S^v_{E_k^+-\delta}Z(\cdot)=B-\delta P(\cdot),
\]
with \[ P(\cdot) := \begin{pmatrix}
z_{11}(\cdot)z_{12}(\cdot)-c z_{11}^2(\cdot)&-c z_{11}(\cdot)z_{12}(\cdot)+z^2_{12}(\cdot)\\
-z_{11}^2(\cdot)&-z_{11}(\cdot)z_{12}(\cdot)
\end{pmatrix}.\]

For $ f:\T^d\to \C $, let $ [f] :=\int_{\T^d} f(\theta) \dif\theta $. Using one step of averaging, one can construct a transformation $\widetilde{Z}$, which conjugates the cocycle $ (\alpha,B-\delta P(\cdot)) $ to $ (\alpha,e^{b_0-\delta b_1}+\delta^2P_1(\cdot)) $ where $$ b_0:=\begin{pmatrix}
		0&c\\0&0
		\end{pmatrix} ,  
		\quad b_1:=\begin{pmatrix}
		[z_{11}z_{12}]-\frac{c}{2}[z_{11}^2]&-c[z_{11}z_{12}]+[z_{12}^2]\\-[z_{11}^2]&-[z_{11}z_{12}]+\frac{c}{2}[z_{11}^2]
		\end{pmatrix}. $$
More precisely, we have the following lemma:
\begin{lemma}\label{lem6.1}
	For any $R\in(0,r]$, let $ D_R:=\sup_{n\in\mathbb{Z}^d} 4\gamma^{-3}|n|^{3\tau}e^{-\frac{R}{2}(2\pi|n|)^\nu}, $ if \\$ 0<\delta<4^{-1}D_R^{-1}|Z|_R^{-2} $, then there exist $ \widetilde{Z}\in G^\nu_{\frac{R}{2}}(\mathbb{T}^d,\mathrm{SL}(2,\R)) $ and $ P_1\in G^\nu_{\frac{R}{2}}(\mathbb{T}^d,\mathrm{gl}(2,\R))$ such that\begin{equation}\label{72}
		\widetilde{Z}(\cdot+\alpha)^{-1}(B-\delta P(\cdot))\widetilde{Z}(\cdot)=e^{b_0-\delta b_1}+\delta^2P_1(\cdot),
		\end{equation}
		with the estimates\begin{equation}\label{73}
		|\widetilde{ Z}-\mathrm{Id}|_{\frac{R}{2}}< 1,\quad|P_1|_{\frac{R}{2}}
	\leq 8(2+D_R)^2|Z|_R^4+\delta^{-1}c^2|Z|_R^2. 		
		\end{equation}
\end{lemma}
\begin{proof}
    The proof follows a similar line as Lemma 6.1 in \cite{leguil2017asymptotics} (see also \cite{cai2021polynomial}), we will give the proof in Appendix \ref{appendixB} for completeness.
\end{proof}

\noindent Thus $ \widetilde{Z} $ is homotopic to identity, we have\[ \rho(\alpha,B-\delta P(\cdot))=\rho(\alpha,e^{b_0-\delta b_1}+\delta^2P_1(\cdot)). \]

In order to show $ |G_k(v)|\leq\delta_1 $ for some $ \delta_1>0 $, by Gap Labelling Theorem, it suffices to show that $ \rho(\alpha, e^{b_0-\delta_1 b_1}+\delta_1^2P_1(\cdot)) >0 $. 

Let $ d(\delta):=\det(b_0-\delta b_1) $. By a direct computation, we get

\begin{align*}\label{77}
	d(\delta) &=-\delta[z_{11}^2]c+\delta^2([z_{11}^2][z_{12}^2]-[z_{11}z_{12}]^2)-\frac{\delta^2}{4}c^2[z_{11}^2]^2 \notag\\
	&=\delta([z_{11}^2][z_{12}^2]-[z_{11}z_{12}]^2)\left( \delta-\frac{[z_{11}^2](c+\frac{1}{4}c^2\delta[z_{11}^2])}{[z_{11}^2][z_{12}^2]-[z_{11}z_{12}]^2}\right) .
	\end{align*}

Let $\tilde{r} =\frac{r_0+r}{2}, \chi=\frac{\tilde{r}-r}{6\tilde{r}}$, and let $ \delta_1=c^{1-\chi} $. 
\begin{claim}
    $ d(\delta_1)\geq \frac{9}{4}c^2 $.
\end{claim}
\begin{proof}
    To get a lower bound of $d(\delta_1)$, we need to use the following lemma in \cite{leguil2017asymptotics}.

\begin{lemma}[\cite{leguil2017asymptotics}]\label{lem6.3}
	For any $ \kappa\in(0,\frac{1}{4}) $, if $ c^{\frac{\kappa}{2}}|Z|_{\T^d} \leq\frac{1}{4} $, we have\begin{align*}
	0<\frac{[z_{11}^2]}{[z_{11}^2][z_{12}^2]-[z_{11}z_{12}]^2}\leq\frac{1}{2}c^{-\kappa},\quad
	[z_{11}^2][z_{12}^2]-[z_{11}z_{12}]^2\geq 8c^{2\kappa}.
	\end{align*}
\end{lemma}
   Reviewing (\ref{3.15}), we get  $c\leq\epsilon_{0}^{\frac{3}{5}}e^{-\frac{r}{1-\chi}|2\pi k|^{{\nu}}}$, $|Z|_{R}\leq D' e^{3R(\pi|k|)^{{\nu}}}$. 
 Then we set $R=\frac{1}{8}\frac{\chi}{1-\chi}r$, for $\epsilon_0$ sufficient small, we have $$ 0<\delta_1\leq\epsilon_0^{\frac{1}{2}}e^{-r|k|^{{\nu}}}<4^{-1}D_{R}^{-1}|Z|_R^{-2} .$$ Hence we can apply Lemma \ref{lem6.1}, and conjugate the system to the cocycle $ (\alpha, e^{b_0-\delta_1 b_1}+\delta_1^2P_1(\cdot)) $.
	
	Since $ c^{\frac{\chi}{2}}|Z|_R\leq \frac{1}{4} $, then we can apply Lemma \ref{lem6.3}, and get
	\[ 
        \frac{[z_{11}^2]c}{[z_{11}^2][z_{12}^2]-[z_{11}z_{12}]^2}\leq\frac{1}{2}\delta_1,
    \]
	hence we have
	\begin{equation}\label{86}
	d(\delta_1)\geq c^{1-\chi}\cdot 8c^{2\chi}\cdot c^{1-\chi}\left( \frac{1}{2}-\frac{1}{8}c^{2-\chi}[z_{11}^2]\right) \geq 
\frac{9}{4}c^2.
	\end{equation}
\end{proof}

	By the expression of $ b_0 $ and $ b_1 $, we have
	\begin{equation}\label{87}
	|b_0-\delta_1 b_1|\leq c+\delta_1(1+c)|Z|^2_{R}\leq\frac{3}{2}c^{1-\chi}|Z|_R^2.
	\end{equation}
	In view of Lemma 8.1 in \cite{hou}, there exists $ \mathcal{P}\in\mathrm{SL}(2,\R) $ with $ |\mathcal{P}|\leq2(\frac{|b_0-\delta_1 b_1|}{\sqrt{d(\delta_1)}})^{\frac{1}{2}} $, such that 
	\[ \mathcal{P}^{-1}e^{b_0-\delta_1 b_1}\mathcal{P}=R_{\sqrt{d(\delta_1)}}. \]
	Combining (\ref{86}) and (\ref{87}), we have
	\[\frac{|b_0-\delta_1 b_1|}{\sqrt{d(\delta_1)}}\leq\frac{\frac{3}{2}c^{1-\chi}|Z|_R^2}{\sqrt{\frac{9}{4}c^2}}\leq c^{-\chi}|Z|_R^2.\]
	According to (\ref{lem3.2}), and in view of (\ref{73}), 
    
    \begin{equation*} 
        \begin{split}
            | \rho(\alpha, e^{b_0-\delta_1 b_1}+\delta_1^2P_1)-\sqrt{d(\delta_1)}| &\leq \delta_1^2|\mathcal{P}|^2|P_1|_{\frac{R}{2}} \\
            &\leq 32(2+D_R)^2 c^{2-3\chi}|Z|_R^6+4c^{3-2\chi}|Z|_R^2.
        \end{split}
    \end{equation*} 
Since 
\[ 32(2+D_R)^2 c^{1-3\chi}|Z|_R^6+4c^{2-2\chi}|Z|_R^2<\frac{1}{2}, \]
which implies that
\[
    \rho(\alpha, e^{b_0-\delta_1 b_1}+\delta_1^2P_1)\geq \sqrt{d(\delta_1)}-| \rho(\alpha, e^{b_0-\delta_1 b_1}+\delta_1^2P_1)-\sqrt{d(\delta_1)}| >0, 
\]
this means $ |G_k(v)|\leq\delta_1\leq\epsilon_0^{\frac{1}{2}}e^{-r|2\pi k|^{{\nu}}} $.

\section{Proof of Theorem \ref{intervalspec}}
In this section, we will prove Theorem \ref{intervalspec} by the following approach used in Takase \cite{takase2021spectra}.
It is known that the integrated density of states (IDS) is $ \frac{1}{2} $-H\"older continuous in small coupling regime.
\begin{theorem}[\cite{Cai2017SharpHC}]\label{Caiao}
    Let $ \alpha\in \mathrm{DC}_d(\gamma,\tau) $, $ v\in C^k(\T^d,\R) $ with $ k\geq 550\tau $. Then there exists $ \epsilon_*=\epsilon_*(\gamma,\tau,d,k) $ such that if $ |v|<\epsilon_* $, then the integrated density of states $ N_{v,\alpha} $ is $ \frac{1}{2}$-H\"older continuous: 
    $$
        |N_{v,\alpha}(E+\epsilon)-N_{v,\alpha}(E-\epsilon)|\leq C_0 \epsilon^{\frac{1}{2}}, \ \forall \epsilon>0,\ \forall E\in\R,
    $$ 
    where $ C_0 $ depends only on $ \gamma,\tau,d $.
\end{theorem}
With the help of Theorem \ref{Caiao}, we have the following lemma.
\begin{lemma}\label{thicklem}$ \lim_{|v_i|\to 0}\tau(\Sigma_i)=+\infty. $
\end{lemma}
\begin{proof}
    Suppose $ 0<|v_i|<\epsilon_0 $, $ \epsilon_0 $ sufficiently small, such that Theorem \ref{thm1.2} and Theorem \ref{Caiao} hold. By the famous Gap-Labelling theorem, every bounded gap can be labelled by $ G_k(v)=[E_k^-,E_k^+] $ with $ k\neq 0 $. Without loss of generality, we only consider the bridge of $ \Sigma_i $ at $ E_k^+ $, denoted by $ C_k^+(v)=[E_k^+,F_k^+] $, by the definition of bridge, $ F_k^+\in\Sigma_i $. Thus 
    $$
        \tau(\Sigma,E_k^+)=\frac{|F_k^+-E_k^+|}{|E_k^+-E_k^-|}.
    $$ 
    Since $ \Sigma_i $ is a Cantor set with no isolated points, we have
    $$
        E_k^+<F_k^+\leq \sup\Sigma_i.
    $$ 
    Case 1. $ F_k^+=\sup\Sigma_i $. Therefore
    $$
        |N_{v,\alpha}(F_k^+)-N_{v,\alpha}(E_k^+)|=|\langle k,\alpha\rangle|_{\T}
    $$ 
    \begin{align*}
        \frac{|F_k^+-E_k^+|}{|E_k^+-E_k^-|}\geq\frac{(\frac{1}{C_0})^{2}|N_{v,\alpha}(F_k^+)-N_{v,\alpha}(E_k^+)|^2}{\epsilon_0^{\frac{1}{2}}e^{-r|2\pi k|^{{\nu}}}}\geq \frac{\gamma^2}{C_0^2\epsilon_0^{\frac{1}{2}}}\frac{1}{e^{-r|2\pi k|^{{\nu}}}|k|^{2\tau}}.
    \end{align*}
    Case 2. $ F_k^+<\sup\Sigma_i $. Then $ F_k^+ $ is the left boundary point of some nonempty gap $ G_{k'}(v)$ with $ \ell(G_{k}(v))<\ell(G_{k'}(v)) $. For any nonempty gap $ G_{k}(v) $, 
    $$
        \zeta(k)=\left[\frac{1}{-r}\log[\epsilon_0^{-\frac{1}{2}}\ell(G_{k}(v))]\right]^{\frac{1}{\nu}}
    $$  
    is well-defined. Thus by the spectral gap estimates, $ \ell(G_{k}(v))\leq\epsilon_0^{\frac{1}{2}}e^{-r|2\pi k|^{{\nu}}} $, we have 
    $$
        \zeta(k)\geq\zeta(k') \text{ and } \zeta(k)\geq\max\{|k|,|k'|\}.
    $$ 
    Therefore
    $$
        |N_{v,\alpha}(F_k^+)-N_{v,\alpha}(E_k^+)|=|\langle k'-k,\alpha\rangle|_{\T}
    $$ 
    \begin{align*}
        \frac{|F_k^+-E_k^+|}{|E_k^+-E_k^-|}\geq\frac{(\frac{1}{C_0})^{2}|N_{v,\alpha}(F_k^+)-N_{v,\alpha}(E_k^+)|^2}{\epsilon_0^{\frac{1}{2}}e^{-r\zeta(k)^{{\nu}}}}\geq \frac{\gamma^2}{C_0^2\epsilon_0^{\frac{1}{2}}}\frac{1}{e^{-r\zeta(k)^{{\nu}}}|2\zeta(k)|^{2\tau}}.
    \end{align*}
    Notice that 
    $$
        \tau(\Sigma)=\inf_u \tau(\Sigma,u)=\inf_k \frac{|F_k^+-E_k^+|}{|E_k^+-E_k^-|}.
    $$ 
    Since $ e^{-r\zeta(k)^{{\nu}}}|2\zeta(k)|^{2\tau}\leq e^{-r|2\pi k|^{{\nu}}}|2k|^{2\tau}\to 0 $ when $ k $ large, thus there is a maximum $ M $ such that $ e^{-r\zeta(k)^{{\nu}}}|2\zeta(k)|^{2\tau}\leq M $, then let $ \epsilon_0\to 0 $, we have
    $
        \lim_{|v_i|\to 0}\tau(\Sigma_i)=+\infty.
    $
\end{proof}
We also need the following lemma in \cite{takase2021spectra}.
\begin{lemma}[\cite{takase2021spectra}]\label{distHaus}
    Let $ A:\mathcal{H}\to\mathcal{H} $ and $ B:\mathcal{H}\to\mathcal{H} $ be bounded self-adjoint operators. Then 
    $$
        |\diam\ \sigma(A)-\diam\ \sigma(B)|\leq 2\mathrm{dist}_{\mathrm{Haus}}(\sigma(A),\sigma(B))\leq 2\lVert A-B\rVert.
    $$ 
\end{lemma}
\noindent Now it suffices to check the conditions of Theorem \ref{astels}. For each set $ \Sigma_i $, by Lemma \ref{distHaus}, one has
$ 
\lim\limits_{|v_i|\to 0}\mathrm{dist}_{\mathrm{Haus}}(\Sigma_i,[-2,2])=0 \text{ and } \lim\limits_{|v_i|\to 0}|\diam \Sigma_i-4|=0, 
$ 
and by Lemma \ref{thicklem}, $ \lim\limits_{|v_i|\to 0}\tau(\Sigma_i)=+\infty $.
Then there exists $ \varepsilon=\varepsilon(\gamma,\tau,\nu,r_0,d)>0 $, such that $ \sum\limits_{i=1}^b\frac{\tau(\Sigma_i)}{\tau(\Sigma_i)+1}\geq 1 $ and 
\begin{equation*} 
    \left\lbrace\begin{aligned}
        &\Gamma(\Sigma_j)\leq \diam(\Sigma_i),\ \ \forall 1\leq j<i\leq b,\\
        &\Gamma(\Sigma_i)\leq \diam(\Sigma_1)+\cdots+\diam(\Sigma_{i-1}),\ \ \forall 2\le i\le b,
    \end{aligned}\right. 
\end{equation*} 
 if $ 0<|v_i|<\varepsilon $ for $ i=1,\cdots,b $, then $ \hat{\Sigma}=\Sigma_1+\Sigma_2+\cdots+\Sigma_b $ is an interval.

\section{Acknowledgments}
We would like to thank Long Li, Qi Zhou and Hongyu Cheng for useful discussions, and Jiangong You for reading the earlier versions of the paper. This work is partially supported by National Key R\&D Program of China (2020YFA0713300) and Nankai Zhide Foundation.
\appendix
\renewcommand{\appendixname}{Appendix~\Alph{section}}
\section{Proof of Lemma \ref{lem3.1}}\label{appendixA}
 Here we prove Lemma \ref{lem3.1} by the following quantitative Implicit Function Theorem:
\begin{theorem}[\cite{berti2006forced}]\label{thm6.1}
	Let $ X,Y,Z $ be Banach space, $ U\subset X $ and $ V\subset Y $ are neighborhoods of $ x_0 $ and $ y_0 $ respectively. Fix $ s,\delta>0 $ and define $ X_s:=\{x\in X ;\|x-x_0\|_X<s\} $ and $ Y_\delta:=\{y\in Y ;\|y-y_0\|_Y<\delta\} $. Let $ \Psi\in C^{1}(U\times V,Z) $. Suppose that $ \Psi(x_0,y_0)=0 $, and that $ D_y\Psi(x_0,y_0)\in\mathcal{L}(Y,Z) $ is invertible. If
	\[ \sup_{\overline{X_s}}\|\Psi(x,y_0)\|_Z\leq\frac{\delta}{2\|D_y\Psi(x_0,y_0)^{-1}\|_{\mathcal{L}(Z,Y)}}, \]\[ \sup_{\overline{X_s}\times\overline{Y_\delta}}\|\mathrm{Id}_Y-D_y\Psi(x_0,y_0)^{-1}D_y\Psi(x,y)\|_{\mathcal{L}(Y,Y)}\leq\frac{1}{2}, \]
	then there exists $ y\in C^{1}(X_s, \overline{Y_\delta}) $such that $ \Psi(x,y(x))=0. $
\end{theorem}

    We construct the nonlinear functional
    \[ 
        \Psi:G_r^\nu(\mathbb{T},\mathrm{sl(2,\R)})\times G_{r}^{\nu\,(nre)}(\eta) \to G_{r}^{\nu\,(nre)}(\eta) 
    \]
    by
    \begin{equation*}
        \Psi(f,Y)=\mathbb{P}_{nre}[\ln \left( e^{-A^{-1}Y(\theta+\alpha)A}e^{f(\theta)}e^{Y(\theta)}\right) ],
    \end{equation*}
    where $ \ln(\cdot) $ is the principal logarithm of a matrix.
        One can easily check that\[ \Psi(0,0)=0,\ \|\Psi(f,0)\|\leq|f|_{r}. \]
        
        By the definition of Fr\'{e}chet derivative, We only need to consider the linear terms of $ \Psi(f,Y+Y')-\Psi(f,Y) $, we have
        \begin{align*}
        \Psi(f,Y+Y'&)-\Psi(f,Y)=\mathbb{P}_{nre}\left[ \ln \left( e^{-A^{-1}(Y(\theta+\alpha)+Y'(\theta+\alpha))A}e^{f(\theta)}e^{Y(\theta)+Y'(\theta+\alpha)}\right)\right.\\
        &\quad-\ln \left.\left( e^{-A^{-1}Y(\theta+\alpha)A}e^{f(\theta)}e^{Y(\theta)}\right)\right]\\
        &=\mathbb{P}_{nre}\left\{\ln \left( e^{-A^{-1}(Y(\theta+\alpha)+Y'(\theta+\alpha))A}e^{f(\theta)}e^{Y(\theta)+Y'(\theta)}\right)\right.\\
        &\quad-\ln \left( e^{-A^{-1}(Y(\theta+\alpha)+Y'(\theta+\alpha))A}e^{f(\theta)}e^{Y(\theta)}\right)\\
        &\quad+\ln \left( e^{-A^{-1}(Y(\theta+\alpha)+Y'(\theta+\alpha))A}e^{f(\theta)}e^{Y(\theta)}\right)\\
        &\quad-\ln \left.\left( e^{-A^{-1}Y(\theta+\alpha)A}e^{f(\theta)}e^{Y(\theta)}\right)\right\} .
        \end{align*}
        By the Baker-Campbell-Hausdorff (BCH) Formula
        \begin{equation} \label{BakerF}
            e^{X}e^Y= e^{X+Y+\frac{1}{2}[X,Y]+\frac{1}{12}([X,[X,Y]]+[Y,[Y,X]])+\cdots}, 
        \end{equation} 
        a direct calculation shows that	
        \begin{align*}
        D_y\Psi(f,Y)(Y')&=\mathbb{P}_{nre}\{-A^{-1}Y'(\theta+\alpha)A+Y'(\theta)+E+F\},
        \end{align*}where $E$ is a sum of terms of at least 2 orders in $-A^{-1}Y(\theta+\alpha)A$, $f$, $Y'$, and $F$ is a sum of terms of at least 2 orders in $-A^{-1}Y'(\theta+\alpha)A$, $f$, $Y$, but both $E$ and $F$ are only 1 order in $Y'$.
        
        Let $ f=0,Y=0 $, thus $E, F$ equal zero, then\[ D_y\Psi(0,0)(Y')=\mathbb{P}_{nre}\{-A^{-1}Y'(\theta+\alpha)A+Y'(\theta)\}. \]
        By a direct computation, which is in analogy with Lemma 1 of \cite{eliasson}, we have\[ \|D_y\Psi(0,0)(Y')\|=|-A^{-1}Y'(\theta+\alpha)A+Y'(\theta)|_{r}\geq\tilde{\eta}|Y'|_{r}, \] where $$\tilde{\eta}:=\frac{\eta^3}{\eta^2+3\|A\|^2\eta+2\|A\|^4}.$$
        So we have \[ \|D_y\Psi(0,0)^{-1}\|\leq \tilde{\eta}^{-1}. \]
         Assume that $ \eta\geq \|A\|^2\epsilon^\frac{1}{9} $, then we have
        \begin{align*}
        \tilde{\eta}&=\frac{\eta^3}{\eta^2+3\|A\|^2\eta+2\|A\|^4}\\
        &\geq\frac{\|A\|^2}{3\left(\frac{\|A\|^2}{\eta}+(\frac{\|A\|^2}{\eta})^2+(\frac{\|A\|^2}{\eta})^3\right)}\\
        &\geq \frac{1}{9}\|A\|^2\epsilon^{\frac{1}{3}}.
        \end{align*}
        Set $ s=\epsilon,\delta=\epsilon^{\frac{1}{2}} $, we have\[  2\|D_y\Psi(0,0)^{-1}\|\cdot\sup_{\overline{X_s}}\|\Psi(x,0)\| \leq 2\tilde{\eta}^{-1}\epsilon\leq\delta. \]
        On the other hand,
        \begin{align*}\label{mistakes}
        \|D_y\Psi(0,0)(Y')-D_y\Psi(f,Y)(Y')\|&=\|\mathbb{P}_{nre}\{E+F\}\|\\
        &\leq \left( 8\|A\|^2|Y|_{r}+8\|A\|^2|Y|_{r}\right) |Y'|_{r}\\
        &\leq 16\|A\|^2\epsilon^{\frac{1}{2}}|Y'|_{r},
        \end{align*}
        so if $\epsilon$ sufficiently small, we have
        \begin{align*}
        \|\mathrm{Id}_Y-D_y\Psi(0,0)^{-1}D_y\Psi(x,y)\|&\leq\|D_y\Psi(0,0)^{-1}\|\|D_y\Psi(0,0)-D_y\Psi(x,y)\|\\
        &\leq \tilde{\eta}^{-1}\times16\|A\|^2\epsilon^{\frac{1}{2}}\leq\frac{1}{2},
        \end{align*}
        By Theorem \ref{thm6.1}, for $ |f|_{r}\leq\epsilon $, and $ \eta\geq \|A\|^2\epsilon^{\frac{1}{9}} $, there exists $ |Y|_{r}\leq\epsilon^{\frac{1}{2}} $ such that $ \Psi(f,Y)=0 $, i.e.\[ \ln \left( e^{-A^{-1}Y(\theta+\alpha)A}e^{f(\theta)}e^{Y(\theta)}\right)=f^{(re)}(\theta), \]which is equivalent to \[ e^{-Y(\theta+\alpha)}(Ae^{f(\theta)})e^{Y(\theta)}=Ae^{f^{(re)}(\theta)}. \]It is easy to check $ |f^{(re)}|_{r}\leq 2\epsilon $, which complete the proof of Lemma \ref{lem3.1}.
        \section{Proof of Lemma \ref{lem6.1}}\label{appendixB}
        Let $G=-\delta B^{-1}P$, we have $\mathrm{tr}(B^{-1}P)=0$, hence $G\in G^\nu_{R}(\mathbb{T}^d,\mathrm{sl}(2,\R))$. Then we can construct $Y\in G^\nu_{\frac{R}{2}}(\mathbb{T}^d,\mathrm{sl}(2,\R))$ (by identifying the Fourier coefficients) such that
        \[B^{-1}Y(\cdot+\alpha)B-Y=G-[G].\]
        A direct calculation shows that\[|\widehat{Y}(n)|\leq 4\gamma^{-3}|n|^{3\tau}|\widehat{G}(n)|,\]
            thus
            \begin{align*}
            |Y|_{\frac{R}{2}}&=\sum_{n\in\mathbb{Z}^d}|\widehat{Y}(n)|e^{\frac{R}{2}(2\pi|n|)^\nu}\\
            &\leq \sum_{n\in\mathbb{Z}^d}4\gamma^{-3}|n|^{3\tau}e^{-\frac{R}{2}(2\pi|n|)^\nu}|\widehat{G}(n)|e^{ R(2\pi|n|)^\nu}\leq D_R\delta|P|_R.
            \end{align*}

        Then for $\widetilde{Z}:=e^Y$, we have
        \[\widetilde{Z}(\cdot+\alpha)^{-1}(B-\delta P(\cdot))\widetilde{Z}(\cdot)=(B-\delta[P])+\widetilde{P}(\cdot),\] where
        \[\widetilde{P}(\cdot):=\sum_{m+n\geq 2}\frac{1}{m!}(-Y(\cdot+\alpha))^m B\frac{1}{n!}
        Y(\cdot)^n-\delta\sum_{m+n\geq 1}\frac{1}{m!}(-Y(\cdot+\alpha))^m P(\cdot)\frac{1}{n!}Y(\cdot)^n. \]
        Obviously, $|\widetilde{ Z}-\mathrm{Id}|_{\frac{R}{2}}\leq 2|Y|_{\frac{R}{2}}<1$. Since $\sum_{m+n=k}\frac{k!}{m!n!}=2^k$, and $|G|_R\leq \delta|P|_R$, we get 
        \[|\widetilde{P}(\cdot)|_{\frac{R}{2}}\leq D_R^2\delta^2|P|_R^2+4D_R\delta^2|P|_R^2\leq (2+D_R)^2\delta^2|P|_R^2. \]
        By a direct calculation, we can see that 
        \[B-\delta[P]=I+(b_0-\delta b_1)+\frac{1}{2}(b_0-\delta b_1)^2-\delta^2 \frac{1}{2} b_1^2.\]
        Let $P_1:=-\frac{1}{2} b_1^2+\delta^{-2}(-\sum_{k\geq 3}\frac{1}{k!}(b_0-\delta b_1)^k+\widetilde{P})$, note that $b_0$ is nilpotent, a direct computation shows (\ref{73}), we finish the proof of Lemma \ref{lem6.1}.

\bibliographystyle{spmpsci}
\bibliography{ultra_referrence}
\end{document}